\documentclass[11pt,a4paper,eqno]{article}
\usepackage{amsmath,amsfonts,amssymb,wasysym,mathrsfs}
\usepackage[latin1]{inputenc}
\usepackage{shapepar}
\usepackage{tikz}
\usepackage{pgfplots,picture} 
\usepgfplotslibrary{fillbetween}
\pgfplotsset{compat=newest, ticks=none}

\usepackage{graphicx}
\usepackage[T1]{fontenc}
\usepackage{frcursive}
\usepackage{enumerate, enumitem}
\usepackage{hyperref}
\usepackage{ifpdf} 
\usepackage{color}
\usepackage{comment}
\DeclareMathAlphabet{\mathpzc}{OT1}{pzc}{m}{it}
\newcommand{\R}{{\mathbb R}}
\newenvironment{frcseries}{\fontfamily{frc}\selectfont}{}

\newcommand{\supp}{\operatorname{supp}}

\newcommand{\N}{{\mathbb N}}

\newcommand{\be}[1]{\begin{equation}\label{#1}}
\newcommand{\ee}{\end{equation}}

\newcommand{\prf}{\par\smallskip\noindent{\sl Proof. \/}}
\newcommand{\finprf}{\unskip\null\hfill$\;\square$\vskip 0.3cm}

\newenvironment{proof}{\prf}{\finprf}
\newtheorem{theorem}{Theorem}[section]
\newtheorem{lemma}[theorem]{Lemma}

\newtheorem{proposition}[theorem]{Proposition}
\newtheorem{remark}[theorem]{Remark}
\newtheorem{definition}[theorem]{Definition}

\numberwithin{equation}{section}
\newcommand{\nc}{\normalcolor}

\def\qed{\,\unskip\kern 6pt \penalty 500
\raise -2pt\hbox{\vrule \vbox to8pt{\hrule width 6pt
\vfill\hrule}\vrule}\par}
\definecolor{darkblue}{rgb}{0.05, .05, .65}
\definecolor{darkgreen}{rgb}{0.1, .65, .1}
\definecolor{darkred}{rgb}{0.8,0,0}
\date{}

\title{Symmetrization results for general nonlocal linear ellipitic and parabolic problems}

\author{Vincenzo Ferone
\footnote{Dipartimento di Matematica e Applicazioni ``Renato Caccioppoli'', Universit\`a degli Studi di Napoli Federico II, Via Cinthia n. 26, Complesso Universitario Monte Sant'Angelo, 80143 Napoli, Italy. \ E-mail: {\tt ferone@unina.it}}\ ,
Gianpaolo Piscitelli \footnote{Dipartimento di Scienze Economiche, Giuridiche, Informatiche e Motorie, Universit\`a degli Studi di Napoli Parthenope, Centro Direzionale, Isola C4, 80143 Napoli, Italy. \ E-mail: {\tt gianpaolo.piscitelli@uniparthenope.it}}\ 
\ and \ Bruno Volzone\footnote{Dipartimento di Scienze Economiche, Giuridiche, Informatiche e Motorie, Universit\`a degli Studi di Napoli Parthenope, Centro Direzionale, Isola C4, 80143 Napoli, Italy. \ E-mail: {\tt bruno.volzone@uniparthenope.it}}
}

\begin{document}

\maketitle

\begin{abstract}
We establish a Talenti-type symmetrization result in the form of mass concentration (\emph{i.e.} integral comparison) for very general linear nonlocal elliptic problems, equipped  with homogeneous Dirichlet boundary conditions.

\noindent In this framework, the relevant concentration comparison for the classical fractional Laplacian can be reviewed as a special case of our main result, thus generalizing the previous results in \cite{ferone2021symmetrization}.  

\noindent
Finally, using an implicit time discretization techniques,  similar results are obtained for the solutions of Cauchy-Dirichlet nonlocal linear parabolic problems.
\end{abstract}

\section{Introduction}
The aim of the present work is to establish some estimates, in the form of mass concentration comparisons, for solutions to general nonlocal elliptic homogeneous Dirichlet problems.\\
To be more specific, let us assume that $\Omega\subset\mathbb R^N$, $N\geq1$, is an open and bounded set and $K, J$ be measurable nonnegative function such that $J\not\equiv 0$ and
\begin{align}
\label{assumption1}
& K(x,y)=K(y,x)\quad \forall x,y\in\R^N,\\ \label{assumption2}
& x\mapsto \int_{\R^N}  K(x,y)\min\{|x-y|^2,1\}dy\in L_{loc}^1(\R^N),\\ \label{assumption3}
& K(x,y)\geq J(x-y).
\end{align}
Let us denote by $\Omega^\sharp$ the ball with the same measure of $\Omega$ and by $J^\sharp$ the Schwartz rearrangement of $J$ (see Section \ref{subsec_rearrangements} for details); the main result of the paper consists in developing new symmetrization techniques in order to achieve  comparison results for equations involving the following nonlocal operator:
\begin{equation}
    \label{operator}
\mathcal L u(x)= \mathrm{P.V.}\int_{\mathbb R^N}K(x,y)(u(x)-u(y))dy,
\end{equation}
where the principal value (P.V.) integral is meant in the sense that
\[
\mathcal L u(x)=\frac 12 \lim_{\varepsilon\to 0^+}\, \int_{\mathbb R^N \setminus B_\varepsilon(x)}K(x,y)(u(x)-u(y))dy.\]
The dual variational interpretation of the operator $\mathcal L$ will be specified in Subsection \ref{sec_the_nonlocl_pb}. We associate to $\mathcal L$ the following, suitable symmetrized operator ${\mathcal L}^{\sharp}$, whose kernel depends on the kernel $J$ introduced in the lower bound \eqref{assumption3}:
\begin{equation}
    \label{operator*}
{\mathcal L}^{\sharp}v(x)=\frac 12\, \mathrm{P.V.}\int_{\mathbb R^N}J^\sharp(x-y)(v(x)-v(y))dy.
\end{equation}

We denote $c_\sharp$ the radially increasing rearrangement of a nonnegative measurable function $c$ (see Section \ref{subsec_rearrangements} for details). The main Theorem of the paper is the following Talenti-type (see \cite{Talenti1}) symmetrization result.

\begin{theorem}
\label{main_thm_stationary}
Let $\mathcal L$ and $\mathcal L^\sharp$ be defined as in \eqref{operator} and \eqref{operator*} with the kernels $K,J$ satisfying \eqref{assumption1}-\eqref{assumption2}-\eqref{assumption3}, $f\in L^2(\Omega)$, $c\in L^\infty (\Omega)$ such that $c(x)\geq 0$ in $\Omega$. If $u$ and $v$ are the solutions of:
\begin{equation}
\label{problem1}
\begin{cases}
\mathcal{L}u+cu=f\ & \text{in }\Omega,\\
u=0 & \text{on }\mathbb R^n \setminus \Omega,
\end{cases}
\end{equation}
and
\begin{equation}
\label{problem2}
\begin{cases}
{\mathcal L}^{\sharp}v+c_\sharp v=f^\sharp \ & \text{in }\Omega^\sharp,\\
v=0 & \text{on }\mathbb R^n \setminus \Omega^\sharp,
\end{cases}
\end{equation}
respectively, then
\begin{equation}
\label{risultato_confronto}
\int_{B_r}u^\sharp(x)\,dx\leq \int_{B_r}v(x)\,dx\quad \forall r>0,
\end{equation}
where $B_r$ is the ball centered in the origin.
Moreover we have the following energy estimate:
\begin{equation}
\label{energy_estimate}
\begin{split}
&\frac 12 \int_{\R^{N}}\int_{\R^{N}}K(x,y) \left(u(x)-u(y)\right)^2dydx+\int_{\R^N} c(x)u^2(x)dx\\ 
&\quad\leq\frac 12\int_{\R^{N}}\int_{\R^{N}}J^\sharp(x-y) \left(v(x)-v(y)\right)^2dydx+\int_{\R^N} c_\sharp(x)v^2(x)dx.
\end{split}\end{equation}
\end{theorem}

For what concerns the nonlocal \emph{parabolic} case, we consider the cylindrical domain of $\R^{N+1}$ given by $\Omega\times [0,T]$ for $T>0$ and give a related comparison result. We will use the following convention: if $f(x,t)$ is defined in $\Omega\times [0,T]$, then $f^\sharp(x,t)$ and $f_\sharp(x,t)$ denotes the symmetrized function with respect to $x$, for $t$ fixed (\emph{i.e.} Steiner symmetrizations of $f$): 
\begin{theorem}
\label{main_thm_evolution}
Let $\mathcal L$ and $\mathcal L^\sharp$ be defined as in \eqref{operator} and \eqref{operator*} with the kernels $K,J$ satisfying \eqref{assumption1}-\eqref{assumption2}-\eqref{assumption3}, 
$c\in L^\infty (\Omega\times (0,T))$ is nonnegative
, $f\in L^2(\Omega\times (0,T))$, $u_0\in L^2(\Omega)$ and $v_0=v^\sharp_0\in L^2(\Omega^\sharp)$ such that $\int_{B_r}u_0^\sharp (x)dx\leq\int_{B_r}v_0 (x)dx$. If $u$ and $v$ are the solutions of:
\begin{equation}
\label{problem1_evo}
\begin{cases}
\displaystyle u_{t}+\mathcal{L}u+cu=f & \text{in }\Omega\times (0,T),\\
u=0 &\text{in } (\R^N\setminus\Omega) \times (0,T),\\
u(x,0)=u_0(x) & \text{in }\mathbb R^N,\\
\end{cases}
\end{equation}
and
\begin{equation}
\label{problem2_evo}
\begin{cases}
\displaystyle v_{t}+{\mathcal L}^{\sharp}v +c_\sharp v =f^\sharp & \text{in }\Omega^\sharp, \\
v =0 &\text{in } (\R^N\setminus\Omega^\sharp) \times (0,T),\\
v(x,0)=v_0(x) & \text{on }\mathbb R^N,\\
\end{cases}
\end{equation}
respectively, then
\begin{equation}
\label{risultato_confronto_evo}
\int_{B_r}u^\sharp(x,t)\,dx\leq \int_{B_r}v(x,t)\,dx\quad \forall r>0 \ \ \forall t\in [0,T].
\end{equation}
\end{theorem}

\noindent {\sc{Main novelties of the paper.}} We mention that a symmetrization result, in the form of mass concentration comparison, for the fractional Laplacian has been already obtained, for example, in \cite{ferone2021symmetrization}, where the singular kernel is then given by 
\[K(y)=\frac{C}{|x|^{N+2s}},\] with $0<s<1$ and $C$ a proper normalization constant. Moreover, analogous comparison results has been also obtained for singular nonlocal elliptic problem \cite{brandolini2022comparison}.\\[0.2cm]
\noindent The main difficulty in proving Theorem \ref{main_thm_stationary} relies with no doubt in the \emph{very} general form of the kernel $K$ described by \eqref{assumption1}-\eqref{assumption2}-\eqref{assumption3}. Indeed, despite of \cite{ferone2021symmetrization,brandolini2022comparison}, the rearrangement estimates in the weak formulation of our problem \emph{can not} yield an explicit expression in radial coordinates in terms of hypergeometric functions, for which the principal features, such as the main monotonicity and asymptotic behaviors, are known. Thus we had to look for new and more flexible techniques which could be adapted to the present general context.\\[0.1cm]
The remarkable features of the new techniques can be identified in the tools used in the proof of Theorem \ref{main_thm_stationary}, which is divided in two relevant steps. In the first one, choosing the classical Talenti's truncature of the weak solution as a test function, we first apply a \emph{Riesz type rearrangement inequality}, in the form described Proposition \ref{riesz_generalized}, directly on the weak formulation of the problem. This step can be seen as an usage of a \emph{nonlocal} P\'olya-Szeg\H o type inequality for the achievement of a suitable energy estimate (see \cite{lions} for a similar approach in the case of local operators). At this stage, an exponential integrability of the involved kernels, derived by the L\'evy property \eqref{assumption3} in Lemma \ref{pesoL1}, is required. Therefore, we use a nonlocal variant of the coarea formula on the super and sub level sets of the solution (see Proposition \ref{coarea_nonlocale}), which turns out to be essential in treating the behavior on the level sets at height $h>0$; particularly, without this property the passage to the limit as $h\rightarrow 0$ looks quite difficult in such general context.\\[0.2cm]
In the second step, we finalize a comparison result. To this aim, we introduce two key functions defined through the convolutions $\Phi_{1}$ and $\Phi_{2}$ of $J^\sharp$ with the characteristic functions of the ball or its complementary set, whose radial monotonicity is achieved in Proposition \ref{prop_monotonicity} by suitable and nontrivial geometric considerations. Besides, the main argument of the comparison result is based on subtle contradiction argument applied to nonlocal estimates derived from the weak formulations of the initial problem \eqref{problem1} and the symmetrized one \eqref{problem2}. In this point, the monotonicity of the functions $\Phi_{i}$ plays an essential role. Finally, we close the argument by invoking a special maximum/minimum principle in Proposition \ref{MaxMin} applied on the lower order term.\\
The proof of the parabolic mass concentration comparison exhibited in Theorem \ref{main_thm_evolution}  is based on the implicit time discretization scheme, employed in \cite{ALTb,alvino2010comparison} for linear problems and firstly introduced in \cite{vazquez}, \cite{VANS05} for symmetrization in nonlinear diffusion equations. More precisely, we reduce the problem \eqref{problem1_evo} to a sequence of elliptic problems, to which we can apply the comparison result \eqref{risultato_confronto}. See also the survey \cite{VANS05}.
\\[0.4 pt]

\noindent {\sc{Some comments on relevant previous results in the literature.}} 
Actually, the effect of symmetrization on the fractional Laplacian operator \begin{equation}
    \label{fractional_lap}
(-\Delta)^su=\gamma(N,s)\,\text{P.V.} \int_{\R^N}\frac{u(x)-u(y)}{|x-y|^{N+2s}}dy
\end{equation}
being $\gamma(N,s)$ a suitabile normalization constant, has already been exploited in \cite{dBVol}, \cite{feostingaV} for fractional elliptic equations and then in \cite{VazVol1}, \cite{VazVol2}, \cite{VazVolSire}, \cite{VOLZNONLINEAR} in the context of nonlocal diffusion equations of porous medium type. In those papers a symmetrization result in terms of mass concentration of the form \eqref{risultato_confronto} is obtained employing the \emph{local} interpretation of \eqref{fractional_lap} in terms of an extension problem settled on an infinite cylinder $\mathcal{C}_{\Omega}=\Omega\times(0,\infty)$, which was established in the classical result by Caffarelli and Silvestre \cite{Caffarelli-Silvestre} and generalized in \cite{Stinga-Torrea}. In such setting, the extra variable $y>0$ is fixed in the symmetrization techniques, therefore the Steiner symmetrization approach can be handled in the extensions problems associated to the nonlocal problems \eqref{problem1} and \eqref{problem2}. Finally, the mass concentration estimate \eqref{risultato_confronto} is established in the limit as $y\rightarrow 0$. We explicitly remark that such an approach is \emph{not} available for our context, since, to our knowledge, no extension problem can be in principle associated to nonlocal operators with general kernels. Moreover, the \emph{direct} symmetrization approach has the notable benefit to highlight the considerable differences with respect the local results, see for instance the explicit counterexamples in \cite{ferone2021symmetrization} for the fractional laplacian operator.\\[0.5pt]
We also mention that in the recent paper \cite{galiano2022} some $L^p$ estimates of solutions of nonlocal elliptic an parabolic problems with \emph{integrable} kernels $K$ are obtained by Talenti's type symmetrization techniques. We remark that these results are consequence of our general Theorems \ref{main_thm_stationary} and \ref{main_thm_evolution}.\\[0.4 pt]

\noindent {\sc{Potential applications of the results and generalizations.}}
The class of equations we are treating appears in several
contexts, and has attracted a lot of interests in different fields, particularly where anomalous diffusions appear. Examples of the central role played by nonlocal operators in the Applied Sciences can be found, for instance, in Probability theory (since they are generators of stochastic L\'evy processes, \emph{i.e.} special stochastic processes with jumps), Fluid mechanics (for example, in the SQG equation) or in 
Mathematical physics (relativistic Schr\"odinger operators or the Boltzmann equation), peridynamics theory (see \emph{e.g.} in \cite{aksoylu2010results, bellido2014existence,du2012analysis,mengesha2013analysis,mengesha2014bond}). For a very rich of more detailed references we refer to the book \cite[Sec. 1.1]{fernandez2023integro} or \cite{ROSOTONSURVEY}.\\[0.2cm]
In particular, we observe that the results established in the present work allows to consider operators of the form $\mathcal L$ which are \emph{stable-like operators}, that is infinitesimal generators of processes associated with  kernels $K$ such that $K(x,y)=\mathsf{K}(x-y)$, where $\mathsf{K}$ is \emph{comparable} to the kernel $|x|^{-N-2s}$ of the fractional Laplacian (also called rough kernels). In this case, the kernel $\mathsf{K}$ needs not to be homogeneous but it satisfies a strong ellipticity condition 
\begin{equation}
    \label{rough_assumptions}
\frac{C_1}{|y|^{N+2s}}\leq \mathsf{K}(y)\leq \frac{C_2}{|y|^{N+2s}},
\end{equation}
for some positive constants $C_1$ and $C_2$, $C_{1}\leq C_{2}$. Such kind of operators are widely studied in the literature, see \cite[Sec. 2.1.6]{fernandez2023integro} and the references therein. The importance of such kernels in the theory is confirmed, for instance by recent important results on obstacle problems \cite{ros2023obstacle}, \cite{ros2023semiconvexity} and the regularity of the free boundaries  \cite{figalli2023regularity}.
\\
To this concern, some possible, future applications of our result might regard the study of decay estimates and the asymptotics related to nonlocal nonlinear diffusion equations with symmetric kernels verifying \eqref{rough_assumptions}, namely nonlinear parabolic equations of the form
\begin{equation*}
u_t+\mathcal L \phi(u)=0,
\end{equation*}
for some increasing odd diffusivity $\phi(t)$: for instance, the choice $\psi(t)=|t|^{m-1}t$ gives rise to a nonlocal equation of porous medium type. Such types of models attracted a lot of attention in the recent literature, see for instance \cite{prova}. Also operators with more general L\'evy kernels satisfying 
upper and lower, weak scaling conditions might be considered: this types of kernels with explicit examples were studied for instance in \cite{bae_kang} (see also \cite{KasMimica}).
Just to give some explicit, interesting examples, we can consider for instance symmetric kernels of the form
\[
\mathsf{K}(y)=\sum_{i=1}^{k}\frac{1}{|y|^{N+2s_{i}}},\quad k\in\mathbf{N}, s_{i}\in (0,1)\text{ for }i=1,...,k:
\] 
in this case, it is clear that the radial kernel in \eqref{assumption3} is just $J(r)=\mathsf{K}(r)$, being $r$ the radial coordinate, and the operators $\mathcal{L}, \mathcal{L}^{\sharp} $ is the sum of fractional Laplacians
\[
\sum_{i=1}^{k}(-\Delta)^{s_{i}}.
\]
We might also take into account kernels of the logarithmic type in the form
\[
\mathsf{K}(y)=\frac{\log^{\varepsilon}(1+|y|)}{r^{N+2}},
\]
for some small $\varepsilon>0$. For more interesting examples of kernels and properties of the related diffusion equations, we refer to the forthcoming paper \cite{QuirGonzavSor}.
\\
We finally observe that our techniques might also contribute to give results in the interesting framework of nonlocal anisotropic equations, see for instance \cite{ROSOTONSURVEY} and results in the related nonlinear diffusion theory \cite{QuirAnis}.\\[0.4 pt]

\noindent {\sc{Organization of the paper.}} The paper is organized as follows. In Section \ref{sec_prelim}, we fix the notations and give some fundamental preliminary results. Section \ref{sec_proof} is entirely devoted to the proof of the  main Theorem \ref{main_thm_stationary} of the paper. In Section \ref{parabolic_sec} we prove the mass concentration comparison results in the parabolic setting.

\section{Notations and Preliminaries}
\label{sec_prelim}

In order to fix the notation and prove the main Theorems of the paper, we need to recall some useful results on symmetrization, the nonlocal problems we are dealing with and some general results  related to the nonlocal operators.

We denote by with $B_r(x_0)$ the open  ball in $\mathbb{R}^N$, centered at $x_0$, of radius $r$ and, sometimes, we put $B_r=B_r(0)$.
The measure of the unit ball is denoted by $\omega_N:=|B_1|$. Furthermore, for any set $E\subseteq \R^N$, we denote by $E^{\sharp}$ the ball of $\mathbb{R}^{N}$ centered at the origin with the same Lebesgue measure as $E$ ($E^{\sharp}=\R^N$ if $|E|=+\infty$). 

\subsection{Rearrangements}\label{subsec_rearrangements}

In this subsection we recall the definiton of rearrangements and some properties which will be used in the following. For a more exhaustive treatment of the argument we refer, for example, to \cite{ChRice}, \cite{Hardy}, \cite{KawB}, \cite{Kesavan}.

Let us consider a real measurable function $f$ on an open set $\Omega\subset\R^N$ and, for any $t\geq 0$, the set 
\[\Omega_f^t=\left\{  x\in\Omega:\left\vert f\left(
x\right)  \right\vert >t\right\}.\]
We assume that the
\emph{distribution function} $\mu_{f}$ of $f$ is such that
\begin{equation}
\label{distribution}
\mu_{f}(t)  :=\left\vert \Omega_f^t\right\vert<+\infty \qquad\text{for every }t
>0,
\end{equation}
We recall that $\mu_f(\cdot)$ is a right-continuous function, decreasing from $\mu_f (0)=|\supp(f)|$ to $\mu_f(+\infty)=0$ as $t$ increases from 0 to $+\infty$. It presents a discontinuity at every value $t$ which is assumed by $|f|$ on a set of positive measure, and, for such a value of $t$, it holds
\[
\mu_f(t^-)-\mu_f(t)=|\{x\in\Omega:\left\vert f\left(
x\right)  \right\vert =t \}|.
\]

For every $t\geq 0$, we put $r_f(t)=\left(\frac{\mu_f(t)}{\omega_N}\right)^\frac 1N$ and it is clear that $(\Omega_f^t)^\sharp=B_{r_f(t)}$. Furthermore, we observe that $r_f(t)$ is also right-continuous. Obviously we put $r_f(t^-)=\left(\frac{\mu_f(t^-)}{\omega_N}\right)^\frac 1N$.

The \emph{one dimensional decreasing rearrangement} of $f$ is
\begin{equation}
    \label{f^*}
f^{\ast}\left(  \sigma\right)  =\sup\left\{ t\geq0:\mu_{f}\left(  t\right)
>\sigma\right\}\qquad\sigma\in\left[0,+\infty \right),
\end{equation}
that is, $f^*$ is the distribution function of $\mu_f$, so it is a right-continuous function.
We stress that if $\mu_f$ is strictly decreasing, then $f^*$ extends to the whole of the half line $[0,+\infty[$ the inverse function of $\mu_f$. In the general case we have that $f^*(\mu_f(t))\le t$, for $t\in[0,+\infty[$, and $\mu_f(f^*(s))\le s$, for $s\in[0,+\infty[$. We also observe that, if $\mu_f(t)$ has a jump, i.e., $\mu_f(t)<\mu_f(t^-)$ for some $t$, then $f^*(s)$ has a flat zone, i.e., $f^*(s)=t,\ \forall s\in[\mu_f(t),\mu_f(t^-)]$ (see Fig. \ref{figuraa}). Similarly, if $\mu_f(t)$ has a flat zone then $f^*(s)$ has a jump.

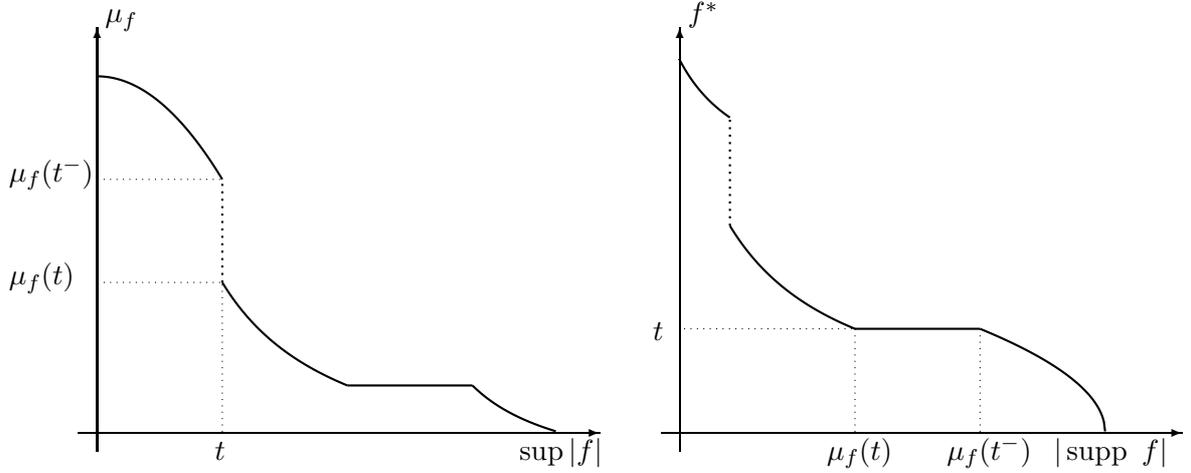
\begin{figure}[ht]\label{figuraa}
\begin{picture}(450,200)
\put(37,-1){\begin{tikzpicture}
\begin{axis}[samples=100,xmin=0,xmax=2.5,
      ymin=0,ymax=2.5,
    axis line style={draw=none},
      tick style={draw=none}, xticklabels={,,},
      yticklabels={,,}]
    \addplot[domain = 1.2:1.8, thick,dotted](0.6,x); 
    \addplot[domain = .33:1.2,dotted](0.6,x); 
    \addplot[domain = 0:.6,dotted](x,1.8); 
    \addplot[domain = 0:.6,dotted](x,1.2); 
    \addplot[domain = 0:{0.6}, thick](x,{2.4-5*(x^2)/3}); 
   \addplot[domain = .6:1.2, thick](x,{18/(25*x)}); 
    \addplot[domain = 1.2:1.8, thick](x,3/5); 
   \addplot[domain = 1.8:2.2, thick](x,{3/(10*x-13)}); 
\end{axis}
\end{tikzpicture}}
\put(227,-1){\begin{tikzpicture}
\begin{axis}[samples=100,xmin=0,xmax=2.5,
      ymin=0,ymax=2.5,
    axis line style={draw=none},
      tick style={draw=none}, xticklabels={,,},
      yticklabels={,,}]
    \addplot[domain = 1.53:2.15,thick,dotted](0.6,x); 
    \addplot[domain = 0.33:.93,dotted](1.2,x); 
    \addplot[domain = 0.33:.93,dotted](1.8,x); 
    \addplot[domain = 0.38:1.2,dotted](x,.93); 
   \addplot[domain = .6:1.2, thick](x,{.33+18/(25*x)}); 
   \addplot[domain = 0.2:{0.6}, thick](x,{1.66+.3/x}); 
    \addplot[domain = 1.8:{2.4}, thick](x,{.33+sqrt((2.4-x)*3/5)}); 
    \addplot[domain = 1.2:1.8, thick](x,0.93); 
\end{axis}
\end{tikzpicture}}
\put(30,20){\vector(1,0){195}}
\put(37,13){\vector(0,1){160}}
\put(248,20){\vector(1,0){195}}
\put(255,13){\vector(0,1){160}}
\put(258,175){$f^*$}
\put(40,175){$\mu_f$}
\put(81,10){$t$}
\put(4,76){$\mu_f(t)$}
\put(4,115){$\mu_f(t^-)$}
\put(310,10){$\mu_f(t)$}
\put(355,10){$\mu_f(t^-)$}
\put(395,10){$|\supp\ f|$}
\put(195,10){$\sup|f|$}
\put(245,55){$t$}
\end{picture}
\caption{A distribution function which presents a discontinuity and a flat zone.}
\end{figure}



If $\Omega$ is bounded, the 
\emph{one dimensional increasing rearrangement} of $f$ is
\[
f_{\ast}\left(\sigma\right)  =f^\ast(|\Omega|-\sigma)\qquad\sigma\in\left(  0,\left\vert \Omega\right\vert \right).
\]
We call the \emph{radially decreasing rearrangement} (or \emph{Schwarz decreasing rearrangement}) of $f$, the function
\[
f^{\sharp}\left(  x\right)  =f^{\ast}(\omega_{N}\left\vert x\right\vert
^{N})\qquad x\in\Omega^{\sharp};
\]
and we call the \emph{radially increasing rearrangement} of $f$, the
function
\[
f_{\sharp}\left(  x\right)  =f_{\ast}(\omega_{N}\left\vert x\right\vert
^{N})\qquad x\in\Omega^{\sharp}.
\]

From the definitions (see in particular \eqref{f^*}), we deduce that $f^*$, $f_*$, $f^\sharp$ and $f_\sharp$ have the same distribution function as $f$, consequently, rearrangements preserve
$L^{p}$
norms, that is, for all $p\in[1,\infty]$:
\begin{equation*}
\|f\|_{L^{p}(\Omega)}=\|f^{\ast}\|_{L^{p}(0,|\Omega|)}=\|f^{\sharp}\|_{L^{p}(\Omega^{\sharp})}.
\end{equation*}

Furthermore, for any couple of measurable functions $f$ and $g$, the classical Hardy-Littlewood inequality \cite{Hardy} holds true
\begin{equation}
\int_{\Omega}\vert f(x)\,  g(x)  \vert dx\leq\int_{0}^{\left\vert \Omega\right\vert}f^{\ast}(\sigma)\,  g^{\ast}(\sigma)  d\sigma=\int_{\Omega^{\sharp}}f^{\sharp}(x)\,g^{\sharp}(x)\,dx\,;
\label{HardyLit}%
\end{equation}
and
\begin{equation}
\int_{\Omega^{\sharp}}f^{\sharp}(x)\,g_{\sharp}(x)\,dx\,=\int_{0}^{\left\vert \Omega\right\vert}f^{\ast}(\sigma)\,  g_{\ast}(\sigma)  d\sigma \le \int_{\Omega}\vert f(x)\,  g(x)  \vert dx.
\label{HardyLit_basso}%
\end{equation}
Now, we recall a generalization of the \emph{Riesz rearrangement inequality} (see \cite[Theorem 2.2]{ALIEB}).
\begin{proposition}
\label{riesz_generalized}
Let $F:\R^{+}\times\R^{+}\rightarrow\R^{+}$ be a continuous function such that $F(0,0)=0$ and
\begin{equation}\label{F}
F(u_{2},v_{2})+F(u_{1},v_{1})\geq F(u_{2},v_{1})+F(u_{1},v_{2})
\end{equation}
whenever $u_{2}\geq u_{1}>0$ and $v_{2}\geq v_{1}>0$.
Assume that $u, v$ are two nonnegative, measurable functions on $\R^{N}$ satisfying \eqref{distribution}, then
\begin{equation}
\int_{\R^{N}}\int_{\R^{N}}F(u(x),v(y))W(ax+by)\,dydx\leq \int_{\R^{N}}\int_{\R^{N}}F(u^{\sharp}(x),v^{\sharp}(y))W^\sharp(ax+by)\,dydx\label{mainRieszineq}
\end{equation}
for any nonnegative function $W\in L^{1}(\R^{N})$ and any choice of nonzero numbers $a$ and $b$.
\end{proposition}

The comparison of mass concentrations \eqref{risultato_confronto} enjoys some nice equivalent formulations
(for the proof we refer to
\cite{Chong}, \cite{ALTa}, \cite{VANS05}).

\begin{proposition}
\label{prop_equiv_conc}
Let $u,v \in L^p(\Omega)$, $p\ge 1$, be two nonnegative function. Then the following are equivalent:

\vskip7pt
\noindent(i)\[
\int_{B_r}u(x)\,dx\leq \int_{B_r}v(x)\,dx\quad \forall r>0,
\]
\vskip7pt

\noindent(ii) for all  nonnegative $\varphi\in L^{p'}(\Omega)$,
$$\int_{\Omega}u(x)\varphi(x)\,dx\leq \int_{\Omega^\sharp}v(x)\varphi^\#(x)\,dx,
$$
\vskip7pt

\noindent(iii) for all convex, nonnegative Lipschitz functions $\Phi:[0,\infty)\rightarrow [0,\infty)$ with $\Phi(0) = 0$ it holds
$$\int_{\Omega}\Phi(u(x))\,dx\leq \int_{\Omega}\Phi(v(x))\,dx.
$$
%
\end{proposition}

From this Lemma it easily follows that if (i) holds,
then
\begin{equation*}
\|f\|_{L^{p}(\Omega)}\leq \|g\|_{L^{p}(\Omega)}\quad \forall p\in[1,\infty].
\end{equation*}

\subsection{The nonlocal problems}
\label{sec_the_nonlocl_pb}
In this Section, we recollect some definitions and properties of the nonlocal problems we are dealing with.

Assume that the nonnegative kernel $K$ satisfies the assumptions \eqref{assumption1}-\eqref{assumption2}-\eqref{assumption3}.
From \cite{felsinger2015dirichlet}, we recall the following definitions and properties of the related functional spaces. We define the Sobolev space $H(\R^N,K)$ as
\[
H(\R^N,K)=\left\{u\in L^{2}(\R^{N}) \text{ such that }[u]_{H(\R^N,K)}<\infty\right\}
\]
where $[u]_{H(\R^N,K)}$ is the \emph{$K$-Gagliardo seminorm of $u$}, that is
\[
[u]_{H(\R^N,K)}=\left(\int_{\R^N}\int_{\R^N}|u(x)-u(y)|^{2} K(x,y)\,dydx\right)^{1/2}.
\]
The space $H(\R^N, K)$ is equipped with the norm 
\[
||u||_{H_\Omega (\R^N, K)}=\left(||u||^2_{L^2(\Omega)}+[u]^2_{H_\Omega (\R^N, K)}\right)^\frac 12.
\]
Moreover, we define the subspace $H_\Omega(\R^N,K)$ of $H(\R^N,K)$ encoding the exterior homogeneous Dirichlet condition for $\Omega$, namely defined by means of
\[
H_\Omega(\R^N,K)=\{ u \in H(\R^N,K)\ : \ u=0 \text{ a.e. } \R^N\setminus\Omega\}.
\]

Under assumptions \eqref{assumption1}-\eqref{assumption2}-\eqref{assumption3}, \cite[Lemma 2.9]{felsinger2015dirichlet} guarantees that the following Poincar\'e-Friedrichs inequality holds:
\begin{equation}\label{PoincFred}
||u||^2_{L^2(\R^N)}\leq C\int_{\R^{N}}\int_{\R^{N}}K(x,y) \left(u(x)-u(y)\right)^2dx\,dy
\end{equation}
where $C>0$ is a proper constant, for all $u\in H_{\Omega}(\R^N, K)$. This allows to equip $H_{\Omega}(\R^N, K)$ of the equivalent norm
\[
||u||_{H_\Omega (\R^N, K)}=[u]_{H_\Omega (\R^N, K)}.
\]
Then the operator $\mathcal{L}$ is defined by duality on the space $H_\Omega (\R^N, K)$ by the identity
\[
\langle \mathcal{L}u,\varphi\rangle=\frac{1}{2}\int_{\R^{N}}\int_{\R^{N}}K(x,y)\left(u(x)-u(y)\right)\left(\varphi(x)-\varphi(y)\right)dydx,
\]
which defines $\mathcal{L}$ as a linear, continuous operator from $H_\Omega (\R^N, K)$ to its dual $(H_\Omega (\R^N, K))^{*}$.


Once this notation fixed, we are able to define the solutions of problem problem \eqref{problem1}.
\begin{definition}
Let $K$ be a nonnegative measurable function satisfying \eqref{assumption1} and \eqref{assumption2}. For all $u,\varphi \in H_{\Omega}(\R^N, K)$, we define 
\begin{equation*}
\mathbb E(u,\varphi;K,c):  =  \frac{1}{2}\int_{\R^{N}}\int_{\R^{N}}K(x,y)\left(u(x)-u(y)\right)\left(\varphi(x)-\varphi(y)\right)dydx+\int_{\Omega} c(x)u(x)\varphi(x)dx.
\end{equation*}
Then $u\in H_\Omega(\R^N,K)$ is called a weak solution of \eqref{problem1} if
\begin{equation}
\label{weak_formulation}
\mathbb E(u,\varphi;K,c)=\int_{\Omega} f(x)\varphi(x)dx
\end{equation}
for all $\varphi \in H_\Omega(\R^N,K)$.
\end{definition}

The existence and uniqueness of the solutions of problem \eqref{problem1} have been extensively studied in many papers; particularly, we will refer to \cite[Proposition 3.4]{felsinger2015dirichlet}.

\begin{proposition}
Let $K$ be a nonnegative measurable function satisying \eqref{assumption1} and \eqref{assumption2}, $c\in L^\infty (\Omega)$ a nonnegative function, $f\in L^2(\Omega)$. Then there exists a unique weak solution $u \in H_{\Omega}(\R^N, K)$ to \eqref{problem1}. Furthermore, $u$ is characterized by the property:
\begin{equation}
    \label{char_min}
\frac{1}{2}\mathbb E(u,u;K,c)-\int_{\Omega} f(x)u(x)dx
=\min_{v \in H_{\Omega}(\R^N, K)}\left\{\frac{1}{2}\mathbb E(v,v;K,c)-\int_{\Omega} f(x)v(x)dx\right\}.
\end{equation}
\end{proposition}

\begin{proof}
It is easily seen that
\[
\mathbb E(u,\varphi;K,c)\leq C_{1} [u]_{H_\Omega(\R^N,K)}\, [\varphi]_{H_\Omega(\R^N,K)}
\]
for a proper positive constant $C_1$. On the other hand, by using the Poincar\'e-Friedrichs inequality \eqref{PoincFred}, we have
\[
\begin{split}
\mathbb E(u,u; K,c)
\geq \frac C2 [u]^2_{H_\Omega(\R^N,K)}.
\end{split}
\]
Therefore by Lax-Milgram Lemma, there is a unique $u\in H_{\Omega}(\R^N,K)$ such that
\[
\mathbb E(u,\varphi;K,c)=\int_{\Omega} f(x)\varphi(x)dx\qquad \forall \varphi \in H_{\Omega}(\R^N,K),
\]
that is the weak formulation \eqref{weak_formulation}. Finally, observing that $\mathbb E(u,\varphi; K,c)$ is symmetric, then Lax-Milgram Lemma also implies \eqref{char_min}.
\end{proof}

Furthermore, the following weak maximum principle is satisfied.
\begin{proposition}\label{prop_weak_max_p}
Let $K$ be a nonnegative measurable function satisying \eqref{assumption1} and \eqref{assumption2}, $c\in L^\infty(\Omega)$ a nonnegative function, $f\in L^2(\Omega)$ and $u \in H_{\Omega}(\R^N, K)$ satisfying
\begin{equation*}
\begin{cases}
\mathcal{L}u+cu=f\ & \text{in }\Omega,\\
u=0 & \text{on }\mathbb R^n \setminus \Omega.
\end{cases}
\end{equation*}
If $f\leq 0$ in $\Omega$, 
then $u \leq 0$ in $\Omega$.
\end{proposition}
\begin{proof}
By choosing
\[\varphi(x)=u^+(x)=
\begin{cases}
u(x)\quad & \text{if }u(x)> 0\\
0 & \text{if }u(x)\leq 0,
\end{cases}\]
in the weak formulation \eqref{weak_formulation}, we have
\begin{equation}
\label{w_max_0}
\begin{split}
\frac 12\int_{\R^{N}}\int_{\R^{N}}K(x,y)\left(u(x)-u(y)\right) & \left(u^+(x)-u^+(y)\right)dydx\\
&+\int_{\Omega} c(x)u(x)u^+(x)dx=\int_{\Omega}f(x)u^+(x)\leq 0.
\end{split}
\end{equation}
We observe that each  term in the l.h.s of \eqref{w_max_0} is nonnegative. Indeed, regarding the first integrand term we observe that
\[
\left(u(x)-u(y)\right)\left(u^+(x)-u^+(y)\right)=
\begin{cases}
\left(u(y)-u(x)\right)^2\quad &\text{if } u(x)\ge 0  \text{ and } u(y)\ge 0,\\
\left(u(x)-u(y)\right)u(x)\quad &\text{if } u(x)\ge 0  \text{ and } u(y)\le 0,\\
\left(u(y)-u(x)\right)u(y)\quad &\text{if } u(x)\le 0  \text{ and } u(y)\ge 0,\\
0 &\text{if } u(x)\le 0  \text{ and } u(y)\le 0;
\end{cases}
\]
the second integrand term is also nonnegative since can be written as $\int_{u\geq 0}c(x)u^2(x)dx$.

Therefore, the l.h.s of \eqref{w_max_0} is null and hence $u^+=0$.
\end{proof}

In order to prove a P\'olya-Szeg\H{o} type inequality, the following exponential integrability of the involved kernel will turn out essential. 

\begin{lemma}\label{pesoL1}
Let $J$ be a nonnegative measurable function such that 
\begin{equation}
\label{int_peso_L1}
x\mapsto \int_{\R^N}  J(x-y)\min\{|x-y|^2,1\}dy\in L_{loc}^1(\R^N).
\end{equation}
Then \begin{equation}
\label{peso_L1}
e^{-\frac{t}{J(x)}}\in L^1(\R^N) \qquad \forall t\geq 0.
\end{equation}
\end{lemma}
\begin{proof}
Firstly, let us observe that the function $\min\{|x|^2,1\}$ is radially increasing. This means that $J^\sharp$ also satisfies \eqref{int_peso_L1}, in view of  
\[
\int_{\R^N}  J(x-y)\min\{|x-y|^2,1\}dy\geq\int_{\R^N}  J^\sharp(x-y)\min\{|x-y|^2,1\}dy.
\]
Therefore, it is sufficient to prove \eqref{peso_L1} when $J$ is replaced by $J^\sharp$, since
\begin{equation}
    \label{exp_norm_preserving}
\|e^{-\frac{t}{J^{\sharp}(\cdot)}}\|_{L^{1}(\R^{N})}=
\|e^{-\frac{t}{J(\cdot)}}\|_{L^{1}(\R^{N})}.
\end{equation}

Let us observe that $e^{-\frac{t}{J^\sharp(x)}}\leq 1$ for any $x\in\R^N$ and $t\geq 0$, that implies $e^{-\frac{t}{J^\sharp(x)}}\in L^1(B_r)$ for any $r>0$. Hence it remains to prove that $e^{-\frac{t}{J^\sharp(x)}}\in L^1(B_r^c)$.

By assumption, we deduce that $\int_{\R^N}  J^\sharp(x-y)\min\{|x-y|^2,1\}dy<+\infty$ for a.e. $x \in\R^N$;  since this quantity is independent by $x$, we have that $\int_{\R^N}  J^\sharp(y)\min\{|y|^2,1\}dy<+\infty$ and hence, for any $R$ large enough, we have
\[
\int_{B_R^c}  J^\sharp(y)dy<+\infty.
\]
Calling $\mathsf{j}$ the radial profile of $J^{\sharp}$, the radial monotonicity of $J^{\sharp}$ gives, for $\rho=|x|>R$,
\[
\int_{R}^{\rho}r^{N-1}\mathsf{j}(r)dr\geq \mathsf{j}(\rho)\,\frac{\rho^{N}-R^{N}}{N}
\]
that is 
\[
\mathsf{j}(\rho)\leq \frac{C}{\rho^{N}-R^{N}}\int_{B_R^c}  J^{\sharp}(y)dy
\]
thus, for some positive constant $C^{\prime}$
\[
e^{-\frac{t}{J^{\#}(x)}}\leq e^{-C^{\prime}t(|x|^N-R^{N})}\qquad\text{ as }|x|\rightarrow+\infty.
\]
Therefore, it remains proven that $e^{-\frac{t}{J^{\sharp}(x)}}\in L^1(B_{R}^c)$ and hence, by \eqref{exp_norm_preserving},
we get the conclusion.
\end{proof}

Now, we state a P\'olya-Szeg\H{o} type inequality, that will be useful in deriving suitable an energy estimates.
 
\begin{proposition}\label{prop_polya}
Let $K,J$ be two nonnegative measurable functions satisfing the assumptions \eqref{assumption1}, \eqref{assumption2}, \eqref{assumption3}. Then, for any $u\in H_{\Omega}(\R^N,K)$, we have that $u^\sharp\in H_{\Omega^\sharp}(\R^N,J^\sharp)$ and
\begin{equation}
\label{K>J*_stima_energia} 
\int_{\R^{N}}\int_{\R^{N}}K(x,y) \left(u(x)-u(y)\right)^2dydx\geq \int_{\R^{N}}\int_{\R^{N}}J^\sharp(x-y) \left(u^\sharp(x)-u^\sharp(y)\right)^2dydx
\end{equation}   
\end{proposition}
\begin{proof}
Let us observe that by \eqref{assumption3}, we directly have
\begin{equation*}
\int_{\R^{N}}\int_{\R^{N}}K(x,y)\left(u(x)-u(y)\right)^2 dydx\ge\int_{\R^{N}}\int_{\R^{N}}J(x-y)\left(u(x)-u(y)\right)^2 dydx.
\end{equation*}
Using \eqref{assumption3} and the arguments in \cite[Section 9]{ALIEB}, we have
\begin{equation}\label{summability_representation_energy} \begin{split}
&\int_{\R^{N}}\int_{\R^{N}}J(x-y)\left(u(x)-u(y)\right)^2 dydx=\int_{0}^{+\infty}\int_{\R^{N}}\int_{\R^{N}}e^{-\frac t{J(x-y)}}\left(u(x)-u(y)\right)^2 dydxdt\\
&=\int_{0}^{+\infty}\int_{\R^{N}}\int_{\R^{N}}e^{-\frac t{J(x-y)}}u^2(x)dydxdt+\int_{0}^{+\infty}\int_{\R^{N}}\int_{\R^{N}}e^{-\frac t{J(x-y)}}u^2(y) dydxdt\\
&\qquad-2\int_{0}^{+\infty}\int_{\R^{N}}\int_{\R^{N}}e^{-\frac t{J(x-y)}}u(x)u(y) dydxdt.
\end{split}\end{equation}
We observe that, in view of Lemma \ref{pesoL1}, the function $e^{-\frac{t}{J(x)}}$ belongs to $ L^1(\R^N)$ for any $t> 0$. This means that we can use inequality \eqref{mainRieszineq} of Proposition \ref{riesz_generalized}, with $F(u,v)=uv$, $a=1$, $b=-1$ and $W(ax+by)=e^{-\frac t{J(x-y)}}$. Observing that $W^\sharp(x)=e^{-\frac t{J^\sharp(x)}}$, for the last term in \eqref{summability_representation_energy} it holds
\begin{equation*}
\int_0^{+\infty}\int_{\R^{N}}\int_{\R^{N}}e^{-\frac t{J(x-y)}}u(x)u(y)dy\,dx\,dt
\leq \int_0^{+\infty}\int_{\R^{N}}\int_{\R^{N}}e^{-\frac t{J^\sharp(x-y)}}u^{\sharp}(x)u^{\sharp}(y)dy\,dx\, dt.
\end{equation*}
Using the same properties, we have that the first two term on the r.h.s. of \eqref{summability_representation_energy} are invariant under rearrangement 
and we get the conclusion \eqref{K>J*_stima_energia}.
\end{proof}

Moreover, we also treat the symmetrized  problem \eqref{problem2}, involving the nonlocal operator \eqref{operator*}.
We observe that the uniqueness of the weak solution of problems in the form \eqref{problem1} implies, for example, that the weak solution of the symmetrized problem \eqref{problem2} has to be radially symmetric. Actually, we can prove the following.

\begin{proposition}
Let $v\in H_{\Omega^\sharp}(\R^N, J^\sharp)$ be the weak solution to problem \eqref{problem2}, then $v=v^\sharp$.
\end{proposition}
\begin{proof}
For all $\varphi \in H_{\Omega^\sharp}(\R^N, J^\sharp)$, let us consider
\begin{equation*}
\mathbb E(v,\varphi;J^\sharp,c_\sharp)  =  \frac{1}{2}\int_{\R^{N}}\int_{\R^{N}}J^\sharp(x-y)\left(v(x)-v(y)\right)\left(\varphi(x)-\varphi(y)\right)dydx+\int_{\Omega^{\sharp} }c_\sharp(x)v(x)\varphi(x)dx\end{equation*} 
By Proposition \ref{prop_polya} and by \eqref{HardyLit_basso}, we have 
\begin{equation}
    \label{E>E}
\mathbb E (v,v; J^\sharp,c_\sharp)\geq \mathbb E (v^\sharp,v^\sharp; J^\sharp,c_\sharp).
\end{equation}
On the other hand, we have
\begin{equation}
\label{v>v}
\int_{\Omega^\sharp}f^\sharp (x)v(x)\,dx\leq \int_{\Omega^\sharp}f^\sharp(x)v^\sharp(x)\,dx.
\end{equation}
Therefore, by \eqref{E>E} and \eqref{v>v}, we have that $v^{\sharp}$ is a minimizer of  the functional 
\[
\mathbb E (v,v; J^\sharp,c_\sharp)-\int_{\Omega^{\sharp}} f^{\sharp}(x)v(x)dx,
\]
thus get the conclusion in view of the uniqueness of the weak solution and the characterization \eqref{char_min}.
\end{proof}

\subsection{A coarea formula and a Maximum Principle}
In this section, we recollect two result that, on the one hand, are very useful to prove the main Theorems, and, on the other hand, have an interest on their own.

Firstly, we prove a coarea formula; we adapt the proof from \cite[Lemma 10]{ambrosio2011gamma} (and \cite{visintin1991generalized}) for the fractional Laplacian to our general case.
\begin{proposition}\label{coarea_nonlocale}
Let $K$ satisfy the assumptions \eqref{assumption1} and \eqref{assumption2}, and $u$ be a nonnegative measurable function such that
\[
\int_{\R^N}\int_{\R^N} K(x,y)|u(x)-u(y)|dydx<+\infty.
\]
Then
\[
\frac{1}{2}\int_{\R^N}\int_{\R^N} K(x,y)|u(x)-u(y)|dydx=\int_0^{\sup u}\int_{u>t}\int_{u\leq t}K(x,y)dydxdt.
\]
\end{proposition}
\begin{proof}
Let us consider the function
\[
t\in [0,\sup u]\mapsto \chi_{u>t}(x)-\chi_{u>t}(y)\in \{-1,0,1\}.
\]
We have
\begin{equation}
\label{relaz_chi_1}
|\chi_{u>t}(x)-\chi_{u>t}(y)|=\left\{
\begin{array}{ll}
0   &   \text{if }t<\min\{u(x),\> u(y)\}\\
\\
0   &   \text{if }t\ge\max\{u(x),\> u(y)\}\\
\\
1   &     \text{if }\min\{u(x),\> u(y)\} \le t<\max\{u(x),\> u(y)\},
\end{array}
\right.
\end{equation}
and
\begin{equation}
    \label{relaz_chi_2}
|\chi_{u>t}(x)-\chi_{u>t}(y)|=\chi_{u>t}(x)\chi_{u\le t}(y)+\chi_{u\le t}(x)\chi_{u>t}(y).
\end{equation}

It follows that
\[
|u(x)-u(y)|=\int_0^{\sup u}|\chi_{u>t}(x)-\chi_{u>t}(y)|dt,
\]
and, furthermore, by Fubini Theorem, we have
\[
\begin{split}
&\int_{\R^N}\int_{\R^N} K(x,y)|u(x)-u(y)|dydx\\
&\quad =\int_\Omega\int_\Omega K(x,y)\int_0^{\sup u}|\chi_{u>t}(x)-\chi_{u>t}(x)|dtdydx\\
&\quad =2\int_0^{u_M}\int_{u>t}\int_{u\le t}K(x,y)dydxdt.
\end{split}
\]
\end{proof}
The following variant of the previous result involving the truncature of a suitable integrable function will be useful in the proof of the main Theorem.
\begin{proposition}\label{coarea_nonlocale2}
Let $K$ satisfy the assumptions \eqref{assumption1} and \eqref{assumption2}, and $u$ be a nonnegative measurable function such that
\[
\int_{\R^N}\int_{\R^N} K(x,y)(u(x)-u(y))^2 dydx<+\infty.
\]
Then, for every $t\ge0$, we have
\begin{equation}
\label{coarea_level}
\begin{split}
\frac{1}{2}\int_{\R^N}\int_{\R^N} K(x,y)(u(x)-u(y))\left((u(x)-t)^+-(u(y)-t)^+\right)dydx\\
=\int_t^{+\infty}\int_{u>\tau}\int_{u\le \tau}K(x,y)(u(x)-u(y))dydxd\tau.
\end{split}
\end{equation}
\end{proposition}
\begin{proof}
Let us consider, for $t\ge0$, the function 
\begin{equation*}
A(t)=\int_{\R^{N}}\int_{\R^{N}}K(x,y)\left(u(x)-u(y)\right)\left((u(x)-t)^+-(u(y)-t)^+\right)dydx,
\end{equation*}
that is finite because
\[\left(u(x)-u(y)\right)\left((u(x)-t)^+-(u(y)-t)^+\right)\le\left(u(x)-u(y)\right)^2.
\]
Using \eqref{relaz_chi_1} and \eqref{relaz_chi_2}, we have
\begin{equation*}
|(u(x)-t)^+-(u(y)-t)^+|=\int_t^{+\infty}\Bigl(\chi_{u\le\tau}(x)\chi_{u>\tau}(y)+\chi_{u>\tau}(x)\chi_{u\le\tau}(y)\Bigr)d\tau.
\end{equation*}
Therefore, we can write
\begin{align*}
A(t)&=\int_t^{+\infty}\int_{\R^{N}}\int_{\R^{N}}K(x,y)|u(x)-u(y)|\Bigl(\chi_{u\le\tau}(x)\chi_{u>\tau}(y)+\chi_{u>\tau}(x)\chi_{u\le\tau}(y)\Bigr)dydxd\tau\\
&=\int_t^{+\infty}\Biggl( \int_{u(x)\le\tau}\int_{u(y)>\tau}K(x,y)|u(x)-u(y)|dydx+\\
&\qquad+\int_{u(x)>\tau}\int_{u(y)\le\tau}K(x,y)|u(x)-u(y)|dydx\Biggr)dt= \\
&= 2\int_t^{+\infty}\int_{u(x)>\tau}\int_{u(y)\le\tau}K(x,y)\bigl(u(x)-u(y)\bigr)dydxd\tau. 
\end{align*}
Therefore \eqref{coarea_level} is proved.
\end{proof}

Now, we recall the following minimum/maximum principle from \cite[Lemma 2.1]{brandolini2022comparison}, whose proof is mainly based on \cite[Theorem 1]{Alvino2008sharp} and \cite[Theorem 3.2]{VazVol1}.

\begin{proposition}\label{MaxMin}
Let $u,v$ be two nonnegative, radial and summable  functions on $B_R$. Let us assume that the function
\[
r\in[0,R]\mapsto\int_{B_r}(u(x)-v(x))dx
\] admits a positive maximum point at $\bar r >0$, that is, 
\begin{equation*}
0<\int_{B_{\bar r}}(u(x)-v(x))dx=\max_{r\in [0,R]}\left(\int_{B_r}(u(x)-v(x))dx\right).
\end{equation*}
Then, if $\mathfrak{h}_1$ is a positive radially increasing function such that $(u(x)-v(x))\mathfrak h_1 (x)$ is summable on $B_{\bar r}$, we have
\begin{equation*}
\int_{B_{\bar{r}}}(u(x)-v(x))\mathfrak{h}_1(x)dx>0.
\end{equation*}
Analogously, let us assume that the function
\[
r\in[0,R]\mapsto \int_{B_R\setminus B_r}(u(x)-v(x))dx
\] 
admits a negative minimum point at $0<\bar{r}<R$, that is,
\begin{equation*}
0>\int_{B_R \setminus B_{\bar r}}(u(x)-v(x))dx=\min_{r\in [0,R]}\left( \int_{B_R \setminus B_{r}}(u(x)-v(x))dx\right).
\end{equation*}
Then, if $\mathfrak{h}_2$ is a positive, radially decreasing function such that $(u(x)-v(x))\mathfrak h_2(x)$ is summable on $B_{R-\bar r}$, we have
\begin{equation*}
\int_{B_R\setminus B_{\bar r}}(u(x)-v(x))\mathfrak{h}_2(x) dx<0.
\end{equation*}
\end{proposition}


\subsection{A Monotonicity Result}
In this section, we give a monotonicity result that is a key tool to prove the main Theorem. Firstly, for any $r>0$, let us set 
\begin{align}
\label{Phi_1}
\Phi_1(x)&=\int_{B_r^c}J^\sharp(x-y)dy=\int_{B_r^c(x)}J^\sharp(y)dy\qquad |x|<r,\\
\label{Phi_2}
\Phi_2(y)&=\int_{B_r}J^\sharp(x-y)dx=\int_{B_r(y)}J^\sharp(x)dx \qquad |y|> r.
\end{align}

We observe that under the assumptions made on $J^\sharp$, the functions $\Phi_1$ and $\Phi_2$ are well defined.
Indeed, as observed in the proof of Lemma \ref{peso_L1}, we have
\[
\int_{\R^N}  J^\sharp(y)\min\{|y|^2,1\}dy<+\infty.
\]
Then, $J^\sharp$ is summable in $B_r^c(x)$ when $|x|<r$ because $B_r^c(x)$ does not contain a suitable neighborhood of the origin. Similarly, $J^\sharp$ is summable in $B_r(y)$ when $|y|>r$ for the same reason.

Our aim is to prove that the function $\Phi_1$ is radially increasing and that the function $\Phi_2$ is radially decreasing with respect to any direction passing through the origin. To gain this result, we first give the following.

\begin{lemma}\label{lemma_distanza}
Let us fix $r>0$, then for any $\rho$, $\rho'$ such that $0<\rho<\rho'$, we have:
\begin{equation}
    \label{ine_monotonicity}
\max_{\overline{B_r(\rho e_1)\setminus B_r(\rho' e_1)}}|y|^{2}\le \min_{\overline{B_r(\rho' e_1)\setminus B_r(\rho e_1)}}|y|^{2},
\end{equation}
where $e_1$ is the first vector of the standard basis.

Furthermore, if $\rho'-\rho\leq 2r$, then
\begin{equation}
    \label{eq_monotonicity}
\max_{\overline{B_r(\rho e_1)\setminus B_r(\rho' e_1)}}|y|^{2}=\min_{\overline{B_r(\rho' e_1)\setminus B_r(\rho e_1)}}|y|^{2}=r^2+\rho\rho'.
\end{equation}
\end{lemma}
\begin{proof}
Let us observe that in the case where $\rho'-\rho>2r$ the balls $B_r(\rho e_1)$ and $B_r(\rho' e_1)$ are disjoint, so:
\begin{equation*}
\max_{\overline{B_r(\rho e_1)\setminus B_r(\rho' e_1)}}|y|^2=\max_{\overline{B_r(\rho e_1)}}|y|^2=(\rho+r)^2,
\end{equation*}
and
\begin{equation*}
\min_{\overline{B_r(\rho' e_1)\setminus B_r(\rho e_1)}}|y|^2=\min_{\overline{B_r(\rho' e_1)}}|y|^2=(\rho'-r)^2.
\end{equation*}
Hence, the inequality \eqref{ine_monotonicity} follows because
\[
(\rho+r)^2-(\rho'-r)^2=(\rho'+\rho)(\rho'-\rho-2r)>0.
\]
So, in order to complete the proof, we have to prove \eqref{eq_monotonicity} when $\rho'-\rho\leq 2r$. 

We firstly observe that
\begin{equation}\label{K1}
y\in \overline{B_r(\rho e_1)\setminus B_r(\rho' e_1)}\ \ \Longrightarrow
\begin{cases}
|y-\rho e_{1}|^2\leq r^2\\
|y-\rho'e_{1}|^2\geq r^2
\end{cases}
\Longrightarrow
\begin{cases}
&|y|^2\leq r^2+2y_1\rho-\rho^2 \\
&|y|^2\geq r^2+2y_1\rho'-\rho'^2.
\end{cases}
\end{equation}
By subtracting each other the two estimates above, we have that \[2y_1(\rho-\rho')-\rho^2+\rho'^2\ge 0\] and consequently that 
 \begin{equation*}
y_1\leq \frac{\rho+\rho'}{2}.
 \end{equation*}
Moreover, from the first inequality in \eqref{K1}, it follows that
\[
|y|^2\leq r^2+\rho\rho'.
\]
Observing that the points $y\in\mathbb\R^N$ such that
\[
|y|^2= r^2+\rho\rho'\qquad\text{and}\qquad y_1= \frac{\rho+\rho'}{2}
\]
belong to $\overline{B_r(\rho e_1)\setminus B_r(\rho' e_1)}$ we have that 
\begin{equation}
    \label{eq_monotonicity1}
\max_{\overline{B_r(\rho e_1)\setminus B_r(\rho' e_1)}}|y|^{2}=r^2+\rho\rho'.
\end{equation}
As regards the minimum in \eqref{eq_monotonicity}, we observe that, for any $y\in\overline{B_r(\rho'  e_1)\setminus B_r(\rho e_1)}$, the inequalities in \eqref{K1} hold with the reverse order relation. By proceeding as before, we get
\[
|y|^2\geq r^2+\rho\rho'\qquad\text{and}\qquad y_1\geq \frac{\rho+\rho'}{2},
\]
and, hence, we have
\begin{equation}
      \label{eq_monotonicity2}
\min_{\overline{B_r(\rho' e_1)\setminus B_r(\rho e_1)}}|y|^{2}=r^2+\rho\rho'.
\end{equation}
Finally, the equality chain \eqref{eq_monotonicity} follows by \eqref{eq_monotonicity1} and \eqref{eq_monotonicity2}.
\end{proof}

\noindent At this point, we are in position to state the monotonicity results for \eqref{Phi_1} and \eqref{Phi_2}.
\begin{proposition}\label{prop_monotonicity}
Let $r>0$, then 
\begin{itemize}
\item $\Phi_1(x)$ is radially increasing for any $|x|<r$;
\item $\Phi_2(y)$ is radially decreasing for any $|y|> r$.
\end{itemize}
    \end{proposition}
\begin{proof}
We firstly observe that $\Phi_i$, $i=1,2$, are radial. Indeed, for any $x_1,x_2$ such that $|x_1|=|x_2|$, there exists a orthogonal rotation matrix $G$, such that $x_2=G(x_1)$ and $\det|G|=1$. By using the change of variables $z=G(y)$, we obtain
\begin{equation*}
\Phi_1(x_1)=\int_{B_r^c}J^\sharp (G(x_1-y))dy=\int_{B_r^c}J^\sharp (x_2-y)dy=\Phi_1(x_2).
\end{equation*}
Similarly, for $\Phi_2$ we have
\begin{equation*}
\Phi_2(y_1)=\int_{B_r}J^\sharp (G(x-y_1))dx=\int_{B_r}J^\sharp (x-y_2)dx=\Phi_1(y_2),
\end{equation*}
and therefore both $\Phi_1$ and $\Phi_2$ are radial.

Being $J^\sharp$ a radial function, it immediately follows that $\Phi_1(x)$ is a radial function, that is $\Phi_1(x)=\Phi_1(|x|)$, for any $|x|<r$.

Therefore, we can study the monotonicity of $\Phi_1$ along a fixed direction passing through the origin. Namely, we will prove that for any $0<\rho <\rho'<r$, we have  $\Phi_1(\rho e_1)\le \Phi_1(\rho' e_1)$, where $e_1$ is the first vector of the standard basis. 
 
Let us observe that since $\rho<\rho'<r$, the sets $B_r(\rho' e_1)\setminus B_r(\rho e_1)$ and $B_r(\rho e_1)\setminus B_r(\rho' e_1)$ are not empty and have the same measure. We have that
\begin{equation}
\label{monotonicity_Phi1}
\begin{split}
\Phi_1(\rho e_1)-\Phi_1(\rho' e_1)&= \int_{B_r^c(\rho e_1)}J^\sharp(y)dy-\int_{B_r^c(\rho' e_1)}J^\sharp(y)dy=\\
&=\int_{B_r(\rho' e_1)\setminus B_r(\rho e_1)}J^\sharp(y)dy-\int_{B_r(\rho e_1)\setminus B_r(\rho' e_1)}J^\sharp(y)dy\\
&\leq \left(\max_{\overline{B_r(\rho' e_1)\setminus B_r(\rho e_1)}} J^\sharp(y)-\min_{\overline{B_r(\rho e_1)\setminus B_r(\rho' e_1)}} J^\sharp(y)\right)|B_r(\rho' e_1)\setminus B_r(\rho e_1)|\leq 0,
\end{split}
\end{equation}
in view of Lemma \ref{lemma_distanza}. Indeed, we notice that since $J^\sharp$ is radially decreasing, the minimum of $J^\sharp$ is achieved in points $\overline{\overline{ y}}\in \overline{B_r(\rho e_1)\setminus B_r(\rho' e_1)}$ such that
\[
|\overline{\overline{ y}}|^{2}=\max_{\overline{B_r(\rho e_1)\setminus B_r(\rho' e_1)}}|y|^{2},
\]
while the maximum of $J^\sharp$ in \eqref{monotonicity_Phi1} is achieved in the point $\bar y\in \overline{B_r(\rho' e_1)\setminus B_r(\rho e_1)}$ such that
\[
|\Bar y|^{2}=\min_{\overline{B_r(\rho' e_1)\setminus B_r(\rho e_1)}}|y|^{2}.
\]

Therefore, the desired monotonicity result follows from Lemma \ref{lemma_distanza} where  we have proven that
\[
\max_{\overline{B_r(\rho e_1)\setminus B_r(\rho' e_1)}}|y|^{2}=\min_{\overline{B_r(\rho' e_1)\setminus B_r(\rho e_1)}}|y|^{2}.
\]
Indeed, from this property follows that for the choice of the points $y_{1}$ and $y_{2}$ as above, we have $|y_{1}|=|y_{2}|$ and thus $J^\sharp(y_{1})=J^\sharp(y_{2})$ and hence the r.h.s. of \eqref{monotonicity_Phi1} vanishes (see also Figure \ref{fig_01_Phi1}).

\begin{figure}[ht]
\begin{picture}(450,200)
\put(100,14){\begin{tikzpicture}[scale=1.1]
\begin{axis}[samples=100,xmin=-210,xmax=210,
      ymin=-210,ymax=300,thick,smooth,no markers,
    axis line style={draw=none},
      tick style={draw=none}, axis equal,xticklabels={,,},
      yticklabels={,,}]
\draw[color=black,fill=white!85!black] (40,0) circle[radius=200];
\draw[color=black,fill=white!85!black
] (100,0) circle[radius=200];
\draw[color=black,fill=white] (72,198) arc
	[
		start angle=98,
		end angle=262,
		x radius=200,
		y radius =200
	] ;
\draw[color=black,fill=white] (72,-198) arc
	[
		start angle=-82,
		end angle=82,
		x radius=200,
		y radius =200
	] ;
\filldraw (40,0) circle (1pt);
\filldraw (0,0) circle (1pt);
\filldraw (100,0) circle (1pt);
\filldraw (72,198) circle (1pt);
\filldraw (72,-198) circle (1pt);
\draw[thick,dotted] (40,0) -- (72,198);
\draw[thick,dotted] (0,0) -- (72,198);
\end{axis}
\end{tikzpicture}}
\put(100,87){\vector(1,0){250}}
\put(207,-4){\vector(0,1){180}}
\put(195,75){$O$}
\put(235,125){$ r$}
\put(215,75){$\rho$}
\put(235,75){$\rho'$}
\end{picture}
\caption{The figure represents the sets $B_r(\rho' e_1)$ and $B_r(\rho e_1)$, when $\rho<\rho'<r$. The symmetric difference between the sets is highlighted. Thicker dots indicate the points where the optimum values of $J^\sharp$ are achieved on the highlighted domain.}
\label{fig_01_Phi1}
\end{figure}
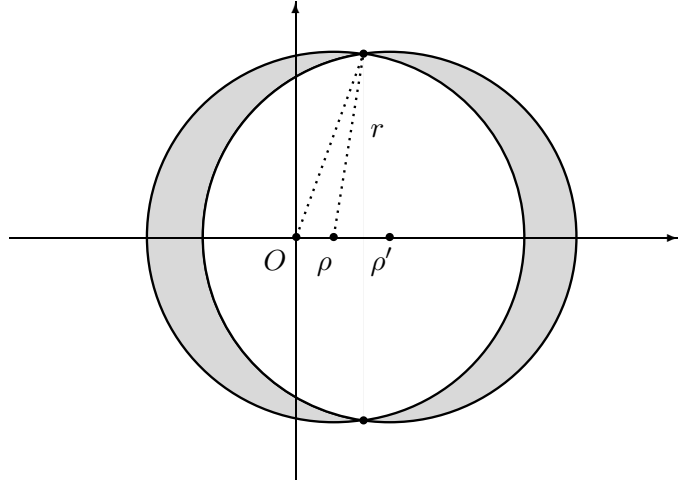

At this stage, we will show that, for any $r<\rho <\rho'<+\infty$, we have $\Phi_2(\rho e_1)\ge \Phi_2(\rho' e_1)$.

We have
\begin{equation*}
\begin{split}
\Phi_2(\rho e_1)-\Phi_2(\rho' e_1)&= \int_{B_r(\rho e_1)}J^\sharp(x)dx-\int_{B_r(\rho' e_1)}J^\sharp(x)dx=\\
&=\int_{B_r(\rho e_1)\setminus B_r(\rho' e_1)}J^\sharp(y)dy-\int_{B_r(\rho' e_1)\setminus B_r(\rho e_1)}J^\sharp(y)dy\\
&\geq \left(\min_{\overline{B_r(\rho e_1)\setminus B_r(\rho' e_1)}} J^\sharp(y)-\max_{\overline{B_r(\rho' e_1)\setminus B_r(\rho e_1)}} J^\sharp(y)\right)|B_r(\rho' e_1)\setminus B_r(\rho e_1)|\ge 0,
\end{split}
\end{equation*}
in view of Lemma \ref{lemma_distanza} (see also Figure \ref{fig_02_Phi2}).
\end{proof}

\begin{figure}[ht]
\begin{picture}(450,200)
\put(100,14){\begin{tikzpicture}[scale=1.1]
\begin{axis}[samples=100,xmin=-210,xmax=210,
      ymin=-210,ymax=300,thick,smooth,no markers, 
    axis line style={draw=none},
      tick style={draw=none}, axis equal,xticklabels={,,},
      yticklabels={,,}]
\draw[color=black,fill=white!85!black] (40,0) circle[radius=200];
\draw[color=black,fill=white!85!black
] (100,0) circle[radius=200];
\draw[color=black,fill=white] (72,198) arc
	[
		start angle=98,
		end angle=262,
		x radius=200,
		y radius =200
	] ;
\draw[color=black,fill=white] (72,-198) arc
	[
		start angle=-82,
		end angle=82,
		x radius=200,
		y radius =200
	] ;
\filldraw (40,0) circle (1pt);
\filldraw (-300,0) circle (1pt);
\filldraw (100,0) circle (1pt);
\filldraw (72,198) circle (1pt);
\filldraw (72,-198) circle (1pt);
\draw[thick,dotted] (-300,0) -- (72,198);
\draw[thick,dotted] (40,0) -- (72,198);
\end{axis}
\end{tikzpicture}}
\put(80,87){\vector(1,0){250}}
\put(102.5,-4){\vector(0,1){180}}
\put(90,75){$O$}
\put(235,125){$ r$}
\put(215,75){$\rho$}
\put(235,75){$\rho'$}
\end{picture}\caption{The figure represents the sets $B_r(\rho' e_1)$ and $B_r(\rho e_1)$, when $r<\rho<\rho'$. The symmetric difference between the sets is highlighted. Thicker dots indicate the points where the optimum values of $J^\sharp$ are achieved on the highlighted domain.}
\label{fig_02_Phi2}
\end{figure}


\section{Proof of Theorem \ref{main_thm_stationary}}
\label{sec_proof}
For any fixed $0\le t<u_{\text{max}}$ and $h>0$, we consider the following test function:
\[
\varphi(x)=\mathcal{G}_{t,h}(u(x)),
\]
where $\mathcal{G}_{t,h}(\theta)$ is the classical truncation
\begin{equation}
\mathcal{G}_{t,h}(\theta)  =\left\{
\begin{array}
[c]{lll}%
h &  & \text{if }\theta > t+h\\
&  & \\
\theta-t\, &  & \text{if }t< \theta \le t+h\\
&  & \\
0 &  & \text{if }\theta \leq t.\text{ }%
\end{array}
\right.\label{truncation}
\end{equation}
For the sake of clarity, we divide the proof into three steps. In the first two claims, we will prove the result supposing that $f\geq 0$; in the third one we will treat the general case; in the fourth one we will give the energy estimate \eqref{energy_estimate}.

{\bf Step 1.} {\it The P\'olya-Szeg\H o type inequality.}

By using  \eqref{truncation} in place of $\varphi$ into the weak formulation \eqref{weak_formulation}, we have
\begin{equation}
\label{eqtest}
\begin{split}
&\frac 12 \int_{\R^{N}}\int_{\R^{N}}K(x,y)\left(u(x)-u(y)\right)\left(\mathcal{G}_{t,h}(u(x))-\mathcal{G}_{t,h}(u(y))\right)dydx\\
&\qquad\qquad\qquad\qquad\qquad\qquad+\int_{\Omega} c(x)u(x)\mathcal{G}_{t,h}(u(x))dx=\int_{\Omega}f(x)\,\mathcal{G}_{t,h}(u(x))\,dx.
\end{split}
\end{equation}

Following \cite[Section 9]{ALIEB} with $K(x,y)$ in place of $|x-y|^{-N-2s}$, we can write
\begin{equation}
\label{summability_representation}
\begin{split}  
&\int_{\R^{N}}\int_{\R^{N}}K(x,y)\left(u(x)-u(y)\right)\left(\mathcal{G}_{t,h}(u(x))-\mathcal{G}_{t,h}(u(y))\right) dydx\\
&=\int_{0}^{+\infty}\int_{\R^{N}}\int_{\R^{N}}e^{-\frac t{K(x,y)}}\left(u(x)-u(y)\right)\left(\mathcal{G}_{t,h}(u(x))-\mathcal{G}_{t,h}(u(y))\right) dydxdt.
\end{split}
\end{equation}
In view of the above representation, we can apply inequality \eqref{F} of Proposition \ref{riesz_generalized}. Indeed, the function $F(u,v)=u^2+v^2-(u-v)(\mathcal{G}_{t,h}(u-\mathcal{G}_{t,h}(v))$ and $W(x)=e^{-\frac t{J(x)}}$ belong to $L^1(\mathbb R^N)$ because of Lemma \ref{pesoL1}. Using $a=1$ and $b=-1$, we have
\begin{equation}
\label{dopo_riesz}
\begin{split}  
&\int_0^{+\infty}\int_{\R^{N}}\int_{\R^{N}}e^{-\frac t{J(x-y)}}\left[u^2(x)+u^2(y)-(u(x)-u(y))\left(\mathcal{G}_{t,h}(u(x))-\mathcal{G}_{t,h}(u(y))\right)\right]dy\,dx\,dt\\
&\leq\int_0^{+\infty}\int_{\R^{N}}\int_{\R^{N}}e^{-\frac t{J^\sharp(x-y)}}\left[(u^{\sharp})^2(x)+(u^{\sharp})^2(y)-(u^{\sharp}(x)-u^{\sharp}(y))\left(\mathcal{G}_{t,h}(u^{\sharp}(x))-\mathcal{G}_{t,h}(u^{\sharp}(y))\right)\right]\,dy\,dx\, dt.
\end{split}
\end{equation}
By using the inequality \eqref{dopo_riesz}, the assumption \eqref{assumption3} and then the representation \eqref{summability_representation}, we deduce that \begin{equation}
\label{Polyatype}   
\begin{split} 
&\int_{\R^{N}}\int_{\R^{N}}J^\sharp(x-y)\left(u^{\sharp}(x)-u^{\sharp}(y)\right)\left(\mathcal{G}_{t,h}(u^{\sharp}(x))-\mathcal{G}_{t,h}(u^{\sharp}(y))\right)dydx    \\
&\le \int_{\R^{N}}\int_{\R^{N}}K(x,y)\left(u(x)-u(y)\right)\left(\mathcal{G}_{t,h}(u(x))-\mathcal{G}_{t,h}(u(y))\right) dydx.
\end{split}
\end{equation}
Then, by joining \eqref{eqtest} and \eqref{Polyatype}, we have
\begin{equation}
\begin{split}  
&\frac{1}{2}\int_{\R^{N}}\int_{\R^{N}}J^\sharp(x-y)\left(u^{\sharp}(x)-u^{\sharp}(y)\right)\left(\mathcal{G}_{t,h}(u^{\sharp}(x))-\mathcal{G}_{t,h}(u^{\sharp}(y))\right)dydx\\
&\qquad\qquad\qquad\qquad\qquad+\int_{\R^N} c(x)u(x)\mathcal{G}_{t,h}(u(x))dx
 \leq \int_{\R^N}f(x)\,\mathcal{G}_{t,h}(u(x))\,dx.\label{mainineq} 
\end{split}
\end{equation}
At this stage, for any $t\ge0$, let us consider the function 
\begin{equation}\label{Psi}
\Psi(t)=\int_{\R^{N}}\int_{\R^{N}}J^\sharp(x-y)\left(u^\sharp(x)-u^\sharp(y)\right)\left((u^\sharp(x)-t)^+-(u^\sharp(y)-t)^+\right)dydx.
\end{equation}
The above integral is finite because
\[\left(u^\sharp(x)-u^\sharp(y)\right)\left((u^\sharp(x)-t)^+-(u^\sharp(y)-t)^+\right)\le\left(u^\sharp(x)-u^\sharp(y)\right)^2.
\]
Furthermore, $\Psi(t)$ is monotone decreasing in view of the fact that, for $t_1<t_2$,
\begin{equation*}
\begin{split}  
\left(u^\sharp(x)-u^\sharp(y)\right)&\left((u^\sharp(x)-t_1)^+-(u^\sharp(y)-t_1)^+\right)\ge\\
&\ge\left(u^\sharp(x)-u^\sharp(y)\right)\left((u^\sharp(x)-t_2)^+-(u^\sharp(y)-t_2)^+\right)
\end{split}
\end{equation*}
A direct computation gives, for $h>0$,
\[\Psi(t)-\Psi(t+h)=\int_{\R^{N}}\int_{\R^{N}}J^\sharp(x-y)\left(u^{\sharp}(x)-u^{\sharp}(y)\right)\left(\mathcal{G}_{t,h}(u^{\sharp}(x))-\mathcal{G}_{t,h}(u^{\sharp}(y))\right)dydx
\]
so, using \eqref{mainineq} and Hardy-Littlewood inequality \eqref{HardyLit}, we have
\begin{equation}\label{inePsi}
\Psi(t)-\Psi(t+h)\le 2\int_{t<u^{\sharp}\le t+h}f^\sharp(x)\bigl(u^\sharp(x)-t\bigr)dx+2h\int_{u^{\sharp}>t+h}f^\sharp(x)dx
\end{equation}
and then $\Psi(t)$ satisfies a Lipschitz condition.
Putting
\begin{align*}
I(h)&=\frac1{2h}\int_{\R^{N}}\int_{\R^{N}}J^\sharp(x-y)\left(u^\sharp(x)-u^\sharp(y)\right)\left(\mathcal{G}_{t,h}(u^\sharp(x))-\mathcal{G}_{t,h}(u^\sharp(y))\right)dydx\\
&=\frac12\bigl(\Psi(t)-\Psi(t+h)\bigr)
\end{align*}
we have, for a.e. $t$,
\begin{equation*}
\lim_{h\rightarrow0^+}I(h)=-\frac12\Psi'(t)
\end{equation*}

On the other hand, by Lemma \ref{coarea_nonlocale2}, we have
\begin{align*}
\Psi(t)= 2\int_t^{+\infty}\int_{u^\sharp(x)>\tau}\int_{u^\sharp(y)\le\tau}J^\sharp(x-y)\bigl(u^\sharp(x)-u^\sharp(y)\bigr)dydxd\tau, 
\end{align*}
It follows
\[\lim_{h\rightarrow0^+}I(h)=-\frac12\Psi'(t)= \int_{u^\sharp(x)>t}\int_{u^\sharp(y)\le t}J^\sharp(x-y)\bigl(u^\sharp(x)-u^\sharp(y)\bigr)dydx.
\]

Taking into account \eqref{mainineq}, \eqref{Psi} and \eqref{inePsi}, it follows
\begin{equation}
\label{insieme}
\begin{split}  
&\frac 12\int_{u^\sharp(x)>t}\int_{u^\sharp(y)\leq t}J^\sharp(x-y)\left(u^{\sharp}(x)-u^{\sharp}(y)\right)dydx\\
&\qquad\qquad\qquad\qquad\qquad+\int_{u^\sharp>t}c_\sharp(x)u^\sharp(x)dx \leq \int_{u^\sharp>t}f^\sharp(x)dx. 
\end{split}
\end{equation}


Hence, for any $r=r_{u}(t)$ the inequality \eqref{insieme} can be written as
\begin{equation}
\label{insieme_Br}
\frac 12 \int_{B_r}\int_{B_r^c}J^\sharp(x-y)(u^{\sharp}(x)-u^{\sharp}(y))dydx+\int_{B_r}c_\sharp(x)u^\sharp(x)dx \leq \int_{B_r}f^\sharp(x)dx.
\end{equation}
Let us observe that if $u^\sharp$ presents a flat zone in correspondence of $r$, that is $r\in[r_{u}(t),r_{u}(t^-)]$, then the inequality \eqref{insieme_Br}
holds true also when $r$ is substituted by $r_{u}(t^-)$, since
\[
\lim_{\varepsilon\to 0^+}r_{u}(t-\varepsilon)=r_{u}(t^-).
\]

{\bf Step 2.} {\it The comparison result.}

As regards the weak solution $v$ to problem \eqref{problem2}, we observe that, integrating the equation on a generic ball $B_r$, we have
\begin{equation}
    \label{insieme_Br_v}
\frac 12\int_{B_r}\int_{B_r^c}J^\sharp(x-y) \left(v(x)-v(y)\right)dydx+\int_{B_r} c_\sharp(x)v(x)dx=\int_{B_r} f^\sharp(x)dx.
\end{equation}
By \eqref{insieme_Br} and \eqref{insieme_Br_v}, for any $r\in[0,+\infty[$ such that $r=r_u(t)$ or $r=r_u(t^-)$ for some $t$, it follows that
\begin{equation}
    \label{insieme_Br_differenza}
\begin{split}
\frac 12\int_{B_r}\int_{B_r^c}J^\sharp(x-y) \left(u^\sharp(x)-u^\sharp(y)\right)dydx-\frac 12\int_{B_r}\int_{B_r^c}J^\sharp(x-y) \left(v(x)-v(y)\right)dydx\\
+\int_{B_r} c_\sharp(x)(u^\sharp(x)-v(x))dx\leq 0.
\end{split}
\end{equation}

At this stage, it remains to prove that
\begin{equation*}
\int_{B_r}u^\sharp(x)dx\leq \int_{B_r}v(x)dx\qquad\forall r\geq 0.
 \end{equation*}

Let us suppose, by contradiction, that the function 
\begin{equation}\label{w}
r\mapsto\int_{B_r}u^\sharp(x)\,dx-\int_{B_r}v(x)\,dx
\end{equation}
has a positive maximum point at $\bar{r}\in (0,R]$, i.e.,
\begin{equation}\label{max>0}
0<\int_{B_{\bar{r}}}u^\sharp(x)\,dx-\int_{B_{\bar{r}}}v(x)\,dx=\max_{r\in[0,R]}\left(\int_{B_r}u^\sharp(x)\,dx-\int_{B_r}v(x)\,dx\right).
\end{equation}

We observe that the function defined in \eqref{w} can be written as a function of the variable $s=\omega_Nr^N$
\[
\int_{B_r}u^\sharp(x)\,dx-\int_{B_r}v(x)\,dx=\int_0^s u^*(\sigma)\,d\sigma-\int_0^s v^*(\sigma)\,d\sigma
\]
and the function
\begin{equation}\label{ws}
s\mapsto\int_0^s u^*(\sigma)\,d\sigma-\int_0^s v^*(\sigma)\,d\sigma
\end{equation}
is convex in any interval where $u^*$ is constant. Thus, if for some $t\in[0,\sup u]$ we have $r_{u}(t)<r_{u}(t^-)$, then the function defined in \eqref{ws} attains its maximum on $\big[\omega_N\big(r_{u}(t)\big)^N,\>\allowbreak \omega_N\big(r_{u}(t^-)\big)^N\big]$ at the boundary. It follows that also  the function defined in \eqref{w} attains its maximum on $[r_{u}(t),\>r_{u}(t^-)]$ at the boundary.

In view of the above observation, we have that the value $\bar r$ can be chosen in such a way that for some $t\in[0,\sup u]$, then either $\bar r=r_{u}(t)$ or $\bar r=r_{u}(t^-)$. 


We also observe that at the point $\bar r$ the function
\begin{equation*}
r\mapsto\int_{B_r^c}u^\sharp(x)\,dx-\int_{B_r^c}v(x)\,dx
\end{equation*}
admits a nonpositive minimum point, that is,
\begin{equation}\label{min<0}
0\ge\int_{B_{\bar r}^c}u^\sharp(x)\,dx-\int_{B_{\bar r}^c}v(x)\,dx=\min_{r\in [0,R]}\left( \int_{B_r^c}u^\sharp(x)\,dx-\int_{B_r^c}v(x)\,dx\right).
\end{equation}
Indeed, for every $r\in [0,R]$,
\begin{align*}
&\int_{B_r^c}u^\sharp(x)\,dx-\int_{B_r^c}v(x)\,dx=\int_{B_R}\bigl(u^\sharp(x)-v(x)\bigr)\,dx-
\int_{B_{r}}\bigl(u^\sharp(x)-v(x)\bigr)\,dx\\
&\qquad\ge\int_{B_R}\bigl(u^\sharp(x)-v(x)\bigr)\,dx-
\int_{B_{\bar r}}\bigl(u^\sharp(x)-v(x)\bigr)\,dx=\int_{B_{\bar r}^c}u^\sharp(x)\,dx-\int_{B_{\bar r}^c}v(x)\,dx
\end{align*}
where the right-hand side in the above inequality is nonpositive because $\bar r$ is a maximum point for the function \eqref{w}.

We have
\begin{equation}
    \label{Juu_insieme}
\begin{split}
        &\int_{B_{\bar r}}\int_{B_{\bar r}^c}J^\sharp(x-y) \bigl(u^\sharp(x)-u^\sharp(y)\bigr)dydx=
\\
&=\int_{|x'|=1}\left(\int_0^{\bar r}\left(\int_{\bar r}^{+\infty}\left( \int_{|y'|=1}J^\sharp(\rho x'-\tau y') \bigl(u^\sharp(\rho)-u^\sharp(\tau)\bigr)dH^{N-1}(y')\right)\right.\right.\\
&\qquad\qquad\qquad\qquad\qquad\qquad\qquad\qquad\qquad\qquad\qquad\quad\tau^{N-1}d\tau\biggr)\rho^{N-1} d\rho\biggr) dH^{N-1}(x')
\\
&=\int_{|x'|=1}\left(\int_0^{\bar r}\left(\Phi_1(\rho)u^\sharp(\rho)\rho^{N-1}\right.\right.\\
&\quad\left.\left.-\rho^{N-1}\int_{\bar r}^{+\infty}\left(\int_{|y'|=1}J^\sharp(\rho x'-\tau y') u^\sharp(\tau)\,dH^{N-1}(y')\right)\tau^{N-1}d\tau\right) d\rho\right)dH^{N-1}(x').
\end{split}
\end{equation}

Let us observe that, in view of \eqref{max>0},
we have
\begin{align*}
&\Phi_1(\rho)u^\sharp(\rho)\rho^{N-1}=\left(-\Phi_1(\rho)\int_\rho^{\bar r}u^\sharp(\sigma)\sigma^{N-1}d\sigma
\right)'+\Phi_1'(\rho)\int_\rho^{\bar r}u^\sharp(\sigma)\sigma^{N-1}d\sigma\\
&\qquad\ge\left(-\Phi_1(\rho)\int_\rho^{\bar r}u^\sharp(\sigma)\sigma^{N-1}d\sigma
\right)'+\Phi'_1(\rho)\int_\rho^{\bar r}v(\sigma)\sigma^{N-1}d\sigma,
\end{align*}
where we have used that $\Phi_1'(\rho)\geq 0$ and the fact that being $\bar{r}$ a maximum point for the function \eqref{w} we have
\[
\int_\rho^{\bar r}u^\sharp(\sigma)\sigma^{N-1}d\sigma
\geq
\int_\rho^{\bar r}v(\sigma)\sigma^{N-1}d\sigma.
\]
Moreover, we also have that
\begin{equation*}
\int_0^{\bar r}\left(-\Phi_1(\rho)\int_\rho^{\bar r}u^\sharp(\sigma)\sigma^{N-1}d\sigma
\right)'d\rho=\Phi_1(0)\int_0^{\bar r}u^\sharp(\sigma)\sigma^{N-1}d\sigma>
\Phi_1(0)\int_0^{\bar r}v(\sigma)\sigma^{N-1}d\sigma.
\end{equation*}
Therefore, by using this inequality in \eqref{Juu_insieme}, we have
\begin{align*}
&\int_{B_{\bar r}}\int_{B_{\bar r}^c}J^\sharp(x-y) \bigl(u^\sharp(x)-u^\sharp(y)\bigr)dydx
\\
&>\int_{|x'|=1}\left(\int_0^{\bar r}\left(\Phi_1(\rho)v(\rho)\rho^{N-1}\right.\right.\\
&\quad\left.\left. -\rho^{N-1}\int_{\bar r}^{+\infty}\left(\int_{|y'|=1}J^\sharp(\rho x'-\tau y') u^\sharp(\tau)\,dH^{N-1}(y')\right)\tau^{N-1}d\tau\right) d\rho\right)dH^{N-1}(x')
\\
&=\int_{B_{\bar r}}\int_{B_{\bar r}^c}J^\sharp(x-y) \bigl(v(x)-u^\sharp(y)\bigr)dydx.
\end{align*}
We can repeat a similar argument for the integral on $B_{\bar r}^c$
\begin{equation}
    \label{Juv_insime}
    \begin{split}
    &\int_{B_{\bar r}}\int_{B_{\bar r}^c}J^\sharp(x-y) \bigl(v(x)-u^\sharp(y)\bigr)dydx=
\\
&=\int_{|y'|=1}\left(\int_{\bar r}^{+\infty}\left(\int_0^{\bar r}\left(\int_{|x'|=1}J^\sharp(\rho x'-\tau y') \bigl(v(\rho)-u^\sharp(\tau)\bigr)dH^{N-1}(x')\right)\right.\right.\\
&\qquad\qquad\qquad\qquad\qquad\qquad\qquad\qquad\qquad\qquad\qquad\rho^{N-1}d\rho\biggr)\tau^{N-1} d\tau\biggr)dH^{N-1}(y')
\\
&=\int_{|y'|=1}\left(\int_{\bar r}^{+\infty}\left(\int_0^{\bar r}\left(\int_{|x'|=1}J^\sharp(\rho x'-\tau y') v(\rho)dH^{N-1}(x')\right)\rho^{N-1}d\rho\right)\tau^{N-1}\right.\\
&\qquad\qquad\qquad\qquad\qquad\qquad\qquad\qquad\qquad\quad-\Phi_2(\tau)u^\sharp(\tau)\tau^{N-1} \biggr)d\tau\biggr)dH^{N-1}(y').
        \end{split}
\end{equation}

Indeed, let us observe that, in view of \eqref{min<0}, we have
\begin{align*}
&\Phi_2(\tau)u^\sharp(\tau)\tau^{N-1}=\left(\Phi_2(\tau)\int_{\bar r}^\tau u^\sharp(\sigma)\sigma^{N-1}d\sigma
\right)'-\Phi_2'(\tau)\int_{\bar r}^\tau u^\sharp(\sigma)\sigma^{N-1}d\sigma\le\\
&\qquad\le\left(\Phi_2(\tau)\int_{\bar r}^\tau u^\sharp(\sigma)\sigma^{N-1}d\sigma
\right)'-\Phi'_2(\tau)\int_{\bar r}^\tau v(\sigma)\sigma^{N-1}d\sigma.
\end{align*}
where we have used that $\Phi_2 '(\tau)\leq 0$, and we also have that
\begin{equation*}
\int_{\bar r}^{R}\left(\Phi_2(\tau)\int_{\bar r}^\tau u^\sharp(\sigma)\sigma^{N-1}d\sigma
\right)'d\rho=\Phi_2(R)\int_{\bar r}^{R}u^\sharp(\sigma)\sigma^{N-1}d\sigma\le
\Phi_2(R)\int_{\bar r}^{R}v(\sigma)\sigma^{N-1}d\sigma.
\end{equation*}
Therefore, by using this inequality in \eqref{Juv_insime}, we have
\begin{equation}
\label{Jvv_insieme}
    \begin{split}    
&\int_{B_{\bar r}}\int_{B_{\bar r}^c}J^\sharp(x-y) \bigl(v(x)-u^\sharp(y)\bigr)dydx\ge
\\
&\ge\int_{|y'|=1}\left(\int_{\bar r}^{+\infty}\left(\int_0^{\bar r}\left(\int_{|x'|=1}J^\sharp(\rho x'-\tau y') v(\rho)dH^{N-1}(x')\right)\rho^{N-1}d\rho\right)\tau^{N-1}\right.\\
&\qquad\qquad\qquad\qquad\qquad\qquad\qquad\qquad\qquad\qquad-\Phi_2(\rho)v(\tau)\tau^{N-1} \biggr)d\tau\biggr)dH^{N-1}(y')
\\
&=\int_{B_{\bar r}}\int_{B_{\bar r}^c}J^\sharp(x-y) \bigl(v(x)-v(y)\bigr)dydx.
    \end{split}
\end{equation}
Hence, by joining \eqref{Juu_insieme}, \eqref{Juv_insime} and \eqref{Jvv_insieme}, we have then proved
\begin{equation*}
\int_{B_{\bar r}}\int_{B_{\bar r}^c}J^\sharp(x-y) \bigl(u^\sharp(x)-u^\sharp(y)\bigr)dydx>
\int_{B_{\bar r}}\int_{B_{\bar r}^c}J^\sharp(x-y) \bigl(v(x)-v(y)\bigr)dydx.
\end{equation*}

Finally, by using Proposition \ref{MaxMin}, we have
\begin{equation*}
\int_{B_{\bar r}}c_\sharp(x)(u^\sharp(x)-v(x))dx>0,
\end{equation*}
that contradicts \eqref{insieme_Br_differenza} at $r=\bar r$.

{\bf Step 3.} {\it The general case.}

If there are no assumptions on the sign of $f$, Theorem \ref{main_thm_stationary}  can be applied to the weak solution $w$ to the equation 
\begin{equation}
    \label{Ltildeu}
\mathcal{L}w+cw=|f|\ \text{ in }\Omega.
\end{equation}
On the other hand, we know that 
\begin{align}
\label{Lu}
\mathcal{L}u+c u=f\ \text{ in }\Omega.
\end{align}
By the linearity of the operator $\mathcal L$, by subtracting \eqref{Ltildeu} to \eqref{Lu} we have
\begin{equation*}
\mathcal{L}(u-w)+c(u-w)=f-|f|\leq 0\ \text{ in }\Omega.
\end{equation*}
By Proposition \ref{prop_weak_max_p}, we deduce that $u-w\leq 0$ in $\Omega$, that is 
\begin{equation}
\label{uleqtildeu}
u\leq w\ \text{ in }\Omega.
\end{equation}
Furthermore, by summing \eqref{Ltildeu}
and \eqref{Lu}, we have
\begin{equation*}
\mathcal{L}(u+w)+c(u+w)=f+|f|\ge 0\ \text{ in }\Omega.
\end{equation*}
Again Proposition \ref{prop_weak_max_p} implies that $u+w\geq 0$ in $\Omega$, that is 
\begin{equation}
\label{-utildelequ}
-w\leq u \ \text{ in }\Omega.
\end{equation}
Therefore \eqref{uleqtildeu} and \eqref{-utildelequ} implies that $|u|\leq w$. Since \eqref{risultato_confronto} holds for $w$ in place of $u$, we have 
\begin{equation*}
\int_{B_r}u^\sharp(x)dx\leq \int_{B_r}w^\sharp(x)dx\leq \int_{B_r}v(x)dx\qquad\forall r\geq 0,
 \end{equation*}
that gives the conclusion.

{\bf Step 4.} {\it The energy estimate \eqref{energy_estimate}.}

By using the weak formulation \eqref{weak_formulation} of problems \eqref{problem1} and \eqref{problem2} with $\varphi=u$ and $\varphi=v$, respectively, and employing Proposition \ref{prop_equiv_conc} (ii), we have
\[
\begin{split}
&\frac 12 \int_{\R^{N}}\int_{\R^{N}}K(x,y) \left(u(x)-u(y)\right)^2dydx+\int_{\R^N} c(x)u^2(x)dx=\int_{\R^N} f(x)u(x)dx\\ 
&\quad\leq\int_{\R^N} f^\sharp(x)v^2(x)dx=\frac 12\int_{\R^{N}}\int_{\R^{N}}J^\sharp(x-y) \left(v(x)-v(y)\right)dydx+\int_{\R^N} c_\sharp(x)v^2(x)dx.
\end{split}
\]
\begin{remark}\label{Extensionconc}
We also observe that from the proof of Theorem \ref{main_thm_stationary} that \eqref{risultato_confronto} still holds when the symmetric rearrangement $f^{\sharp}$ in problem \eqref{problem2} is replaced by a radially decreasing function $g$ \emph{more concentrated than $f$}, \emph{i.e.} when
\[
\int_{B_r}f^\sharp (x)dx\leq\int_{B_r}g(x)dx\qquad \forall r>0.
\]
\end{remark}

\begin{remark} The main Theorem \ref{main_thm_stationary} gives us the possibility to transfer the study of the $L^p$ regularity scale of the solution $u$ to \eqref{problem1} to the same regularity for the solution $v$ to the radial problem \eqref{problem2}. For instance, assume that $\mathcal{L}$ is a stable-like operator, whose kernel $K(x,y)=\mathsf{K}(|x-y|)$, where $\mathsf{K}$ satisfies \eqref{rough_assumptions}. Then $\mathcal{L}^{\sharp}=(-\Delta)^{s}$ and we can take the advantage that $v$ can be written in the integral form in terms of the explicit Green function of the fractional Laplacian on the ball, see the proof of \cite[Th. 3.2]{ferone2021symmetrization}. More precisely, we have that for any $N\geq2$, $f\in L^{p}(\Omega)$, with $p\geq 2N/(N+2s)$, the following properties hold:
\begin{enumerate}
\item if $p<N/(2s)$ then $u\in L^{q}(\Omega)$, with
\begin{equation*}
q=\frac{Np}{N-2sp}
\end{equation*}
and there exists a constant $C$ such that:
\[
\|u\|_{L^{q}(\Omega)}\leq C\|f\|_{L^{p}(\Omega)};
\]
\item if $p>N/(2s)$ then $u\in L^{\infty}(\Omega)$ and there exists a constant $C$ such that:
\[
\|u\|_{L^{\infty}(\Omega)}\leq C\|f\|_{L^{p}(\Omega)};
\]
\item if $p=N/(2s)$, then $u\in L_{\Phi_{p}(\Omega)}$ and there exists a constant $C$ such that:
\[
\|u\|_{L_{\Phi_{p}(\Omega)}}\leq C\|f\|_{L^{p}(\Omega)},
\]
where $L_{\Phi_{p}(\Omega)}$ is the Orlicz space generated by the $N$-function
\[
\Phi_{p}(t)=\exp(|t|^{p^{\prime}})-1.
\]
\end{enumerate}
Similar arguments can be used also in case $N=1$. 
\end{remark}

\section{Concentration comparisons for linear parabolic problems}\label{parabolic_sec}
We extend our results to the parabolic case, as stated in \eqref{risultato_confronto_evo}, by using the method of discretization in time by induction along with the implicit Euler method \cite{lions1969quelques}.

Moreover, to give a weak formulation of problem \eqref{problem1_evo}, we introduce the following Banach space
\[
\mathcal{W}(0,T;H_{\Omega}(\R^N,K)):=\{ u\in L^2(0,T; H_{\Omega}(\R^N,K)) \ : \ u_t \in L^2(0,T;H_\Omega (K,\R^N)')\},
\]
endowed with the norm
\[
||u||_{\mathcal W(0,T;H_{\Omega}(\R^N,K))}=\int_0^T||u(x;t)||^2_{H_{\Omega}(\R^N,K)}dt+\int_0^T||u_t(x;t)||^2_{H_{\Omega}(\R^N,K)^{*}}dt.
\]
Once fixed the notations, we give the following definition.
\begin{definition}\label{definition_evo}
Let $u_0\in L^2(\Omega)$, $c\in L^\infty(\Omega\times (0,T))$ be a nonnegative function, $f\in L^2(\Omega\times (0,T))$ and $u_0\in L^2(\Omega)$.  We say that $u\in \mathcal{W}(0,T;H_{\Omega}(\R^N,K))$ is a weak solution of \eqref{problem1_evo} if for every $\varphi\in H_\Omega(\R^N,K)$, we have
\[
\begin{split}
<u_{t}(t),\varphi>_{H_{\Omega}(\R^N,K)^{'}}  +\frac 12 \int_{\R^{N}}&\int_{\R^{N}} K(x,y)  \left(u(x,t)-u(y,t)\right)\left(\varphi(x)-\varphi(y)\right)dx\,dy\\
&+\int_{\Omega} c(x,t)u(x,t)\varphi(x)dx=\int_{\Omega}f(x,t)\varphi (x)dx\quad \text{ a.e. }t\in (0,T),
\end{split}
\]
and $u(0,x)=u_0(x)$ a.e. in $\Omega$.
\end{definition}
At this stage, we give the proof of the main Theorem in the parabolic case.\\[0.4pt]

\noindent {\it Proof of Theorem \ref{main_thm_evolution}. }
Theorem \cite[Th.5.3]{felsinger2015dirichlet} assures the existence and the uniqueness of problems \eqref{problem1_evo} and \eqref{problem2_evo} in the sense of Definition \ref{definition_evo}. Our aim now is first to introduce the \emph{implicit time discretization scheme}, which shall provides the existence of suitable approximating solutions, to which the elliptic concentration comparison Theorem \eqref{main_thm_stationary} can be applied. The main weak convergences in the approximation procedures will allow to pass to the limit in the final mass concentration estimate.
We divide the proof in three steps.\\ [0.5cm]
{\bf Step 1.} {\it The implicit time discretization scheme.}\\
\noindent Let us fix $N\in \N$ define $\Delta t:=\frac{T}{N}$, $t_n:=n\Delta t$, for $n=0,...,N$, and
\[
c_{n}(x):=\frac 1 {\Delta t}\int_{t_{n}}^{t_{n+1}}c(x,\tau)d\tau\qquad
f_{n}(x):=\frac 1 {\Delta t}\int_{t_{n}}^{t_{n+1}}f(x,\tau)d\tau\qquad x\in \Omega,\ \  n= 0,...,N-1.
\]
Let us consider the problem where the time derivative $u_{t}$ is replaced by a difference quotient
\begin{equation}
    \label{problem1_discr}
\begin{cases}
\mathcal Lu_{n+1}+c_{n} u_{n+1}+\dfrac{u_{n+1}}{\Delta t}=f_{n}+\dfrac{u_{n}}{\Delta t}\quad & \text{in }\Omega\\
u_{n+1}\in  H_\Omega(\R^N,K),
\end{cases}
\end{equation}
where $n=0,...,N-1$ and $u_0=u_0(x)$ is given in the initial conditions.
Moreover, let us set 
\[
d_{n}=(c_{n})_{\sharp}\quad \text{in }\Omega\\
\]
and
\[
g_{n}=f_{n}^{\sharp}\quad \text{in }\Omega\\
\]
and consider the corresponding symmetrized problems
\begin{equation}
    \label{problem2_discr}
\begin{cases}
\mathcal L^\sharp v_{n+1}+d_{n} v_{n+1}+\dfrac{v_{n+1}}{\Delta t}=g_{n}+\dfrac{v_{n}}{\Delta t}\quad & \text{in }\Omega^{\sharp}\\
v_{n+1}\in  H_{\Omega^{\sharp}}(\R^N,K),
\end{cases}
\end{equation}
with $v_0=u_{0}^\sharp(x)$.
Now we define the piecewise constant interpolations
\[
u_{N}(x,t)=\sum_{n=0}^{N-1}u_{n+1}(x)\chi_{[t_{n},t_{n+1}]}(t),\quad v_{N}(x,t)=\sum_{n=0}^{N-1}v_{n+1}(x)\chi_{[t_{n},t_{n+1}]}(t),
\]

\[
c_{N}(x,t)=\sum_{n=0}^{N-1}c_{n}(x)\chi_{[t_{n},t_{n+1}]}(t),\quad d_{N}(x,t)=\sum_{n=0}^{N-1}d_{n}(x)\chi_{[t_{n},t_{n+1}]}(t),
\]

\[
f_{N}(x,t)=\sum_{n=0}^{N-1}f_{n}(x)\chi_{[t_{n},t_{n+1}]}(t),\quad g_{N}(x,t)=\sum_{n=0}^{N-1}g_{n}(x)\chi_{[t_{n},t_{n+1}]}(t).
\]

{\bf Step 2.} {\it Convergence of the approximating sequences.}\\
We show that $\{u_N\}_{N\in\N}$ and $\{v_N\}_{N\in\N}$ converge to the weak solutions $u$ and $v$ of problems \eqref{problem1_evo} and \eqref{problem2_evo}. Though the argument is rather classical and is based on achieving suitable discrete energy estimates, we reproduce here the main steps for the sake of completeness. We write problem \eqref{problem1_discr} for $n=k$, $n\leq N$ and test with $u_{k+1}$, in order to obtain
\[
\frac{\Delta t}{2}||u_{k+1}||_{H_\Omega (\R^N, K)}^{2}+(u_{k+1}-u_{k},u_{k+1})_{L^{2}(\Omega)}+\int_{\Omega}c_{k}u_{k+1}^{2}dx=\Delta t\,(f_{k},u_{k+1})_{L^{2}(\Omega)}.
\]
Since $c_{k}$ is nonnegative, summing the previous equality on $k$ from $0$ to $n$, we find
\[
\frac{\Delta t}{2}\sum_{k=0}^{n}||u_{k+1}||_{H_\Omega (\R^N, K)}^{2}+\sum_{k=0}^{n} (u_{k+1}-u_{k},u_{k+1})_{L^{2}(\Omega)}\leq \Delta t \sum_{k=0}^{n} \,(f_{k},u_{k+1})_{L^{2}(\Omega)}.
\]
Using the identity
\[
(u_{k+1}-u_{k},u_{k+1})_{L^{2}(\Omega)}=\frac{1}{2}\left(\|u_{k+1}-u_{k}\|_{L^{2}(\Omega)}^{2}+\|u_{k+1}\|_{L^{2}(\Omega)}^{2}-\|u_{k}\|_{L^{2}(\Omega)}^{2}\right), \]
we can write the previous estimate in the following form
\begin{equation}\label{firstdiscret}
\sum_{k=0}^{n} \|u_{k+1}-u_{k}\|_{L^{2}(\Omega)}^{2}+\|u_{n+1}\|_{L^{2}(\Omega)}^{2}-\|u_{0}\|_{L^{2}(\Omega)}^{2}+\Delta t\, \sum_{k=0}^{n} ||u_{k+1}||_{H_\Omega (\R^N, K)}^{2}\leq 2\Delta t \sum_{k=0}^{n} \,(f_{k},u_{k+1})_{L^{2}(\Omega)}.
\end{equation}
On the other hand, using Poincar\'e and Young inequalities we have
\[
(f_{k},u_{k+1})_{L^{2}(\Omega)}\leq \frac{1}{2} ||u_{k+1}||_{H_\Omega (\R^N, K)}^{2}+ C\frac{\|f_{k}\|^{2}_{L^{2}(\Omega)}}{2},
\]
for some constant $C>0$, hence estimate \eqref{firstdiscret} leads 
\begin{equation}\label{firstdiscret2}
\sum_{k=0}^{n} \|u_{k+1}-u_{k}\|_{L^{2}(\Omega)}^{2}+\|u_{n+1}\|_{L^{2}(\Omega)}^{2}+\frac{1}{2}\Delta t\, \sum_{k=0}^{n} ||u_{k+1}||_{H_\Omega (\R^N, K)}^{2}\leq \|u_{0}\|_{L^{2}(\Omega)}^{2}+ C\Delta t \sum_{k=0}^{N-1} \| f_{k}\|_{L^{2}(\Omega)}^{2}.
\end{equation}
Now we introduce the piecewise linear interpolation
\[
\hat{u}_{N}(x,t)=\sum_{n=0}^{N-1}\left(u_{n}(x)+\frac{t-t_{n}}{\Delta t}(u_{n+1}(x)-u_{n}(x))\right) \chi_{[t_{n},t_{n+1}]}(t).
\]
It is clear that $u_{N}\in L^{2}(0,T;H_\Omega(\R^N,K))$, while $\hat{u}_{N}\in C([0,T];H_\Omega(\R^N,K))$. Moreover
\[
\max_{t\in [0,T]}\|\hat{u}_{N}(t)\|_{L^{2}(\Omega)}^{2}=\max_{0\leq n\leq N}\|u_{n}\|_{L^{2}(\Omega)}^{2},\quad \int_{0}^{T}\|u_{N}(t)\|^{2}_{H_\Omega(\R^N,K)}dt=\Delta t\sum_{n=1}^{N}\|u_{n}\|^{2}_{H_\Omega(\R^N,K)} ,
\]
hence \eqref{firstdiscret2} gives
\begin{equation}\label{secondiscret}
\max_{t\in [0,T]}\|\hat{u}_{N}(t)\|_{L^{2}(\Omega)}^{2}+\frac{1}{2}\int_{0}^{T}\|u_{N}(t)\|^{2}_{H_\Omega(\R^N,K)}dt\leq C\left( \|u_{0}\|_{L^{2}(\Omega)}^{2}+ \Delta t \sum_{k=0}^{N-1} \| f_{k}\|_{L^{2}(\Omega)}^{2}\right).
\end{equation}
Moreover, an explicit calculation gives
\[
\int_{0}^{T}\|\hat{u}_{N}(t)-u_{N}(t)\|_{L^{2}(\Omega)}^{2}dt=\frac{\Delta t}{3}\sum_{n=0}^{N-1}\|u_{n+1}-u_{n}\|_{L^{2}(\Omega)}^{2}
\]
and using \eqref{firstdiscret2} again we find
\begin{equation}\label{thirdiscret}
\int_{0}^{T}\|\hat{u}_{N}(t)-u_{N}(t)\|_{L^{2}(\Omega)}^{2}dt\leq C\,\Delta t \left( \|u_{0}\|_{L^{2}(\Omega)}^{2}+ \Delta t \sum_{k=0}^{N-1} \| f_{k}\|_{L^{2}(\Omega)}^{2}\right).
\end{equation}
Observe moreover that by the integral Young inequality we find
\[
\Delta t \sum_{n=0}^{N-1}\|f_{n}\|_{L^{2}(\Omega)}^{2}=\int_{0}^{T}\|f_{N}(t)\|_{L^{2}(\Omega)}^{2}dt\leq\int_{0}^{T}\|f(t)\|_{L^{2}(\Omega)}^{2}dt=\|f\|_{L^{2}(\Omega\times (0,T))}^{2}.
\]
Hence estimates \eqref{secondiscret} and \eqref{thirdiscret} imply
\begin{equation}\label{fourest}
\|\hat{u}_{N}\|_{L^{\infty}(0,T;L^{2}(\Omega))}+\|u_{N}\|_{L^{2}(0,T;H_\Omega(\R^N,K))}\leq C,\quad \|\hat{u}_{N}(t)-u_{N}(t)\|_{L^{2}(\Omega\times (0,T))}\leq C\sqrt{\Delta t}.
\end{equation}
From \eqref{fourest} we have in particular that, up to subsequences,
\begin{equation}\label{weakast}
\hat{u}_{N}\overset{\ast}{\rightharpoonup} u \quad\text{in }L^{\infty}(0,T;L^{2}(\Omega)),
\end{equation}
\[
u_{N}\rightharpoonup z \quad\text{weakly in }L^{\infty}(0,T;H_\Omega(\R^N,K))).
\]
and by the second estimate in \eqref{fourest} we find $u=z$, therefore
\begin{equation}\label{weakconv}
u_{N}\rightharpoonup u \quad\text{weakly in }L^{2}(0,T;H_\Omega(\R^N,K))).
\end{equation}
Now we observe that problems \eqref{problem1_discr} can be rewritten in the following unified form
\begin{equation}\label{approximeq}
\hat{u}_{N}^{\prime}(t)+\mathcal{L}u_{N}(t)+c_{N}(t)u_{N}(t)=f_{N}(t),\quad \text{in }(0,T).
\end{equation}
Observe that from \eqref{weakast} we have
\[
\hat{u}_{N}^{\prime}\rightarrow u^{\prime}\text{ in }\mathcal{D}^{\prime}(0,T; H_\Omega(\R^N,K))^{\prime})
\]
while from \eqref{weakconv}
\[
\mathcal{L}u_{N}\ \rightharpoonup \mathcal{L}u \quad\text{weakly in }L^{2}(0,T;H_\Omega(\R^N,K)^{\prime})
\]
and it is not difficult to show that
\[
c_{N}\rightarrow c,\quad f_{N}\rightarrow f\text{ strongly in }L^{2}(\Omega\times (0,T)).
\]
Passing to the limit as $N\rightarrow\infty$ in \eqref{approximeq} we find that $u$ satisfies
\[
u^{\prime}(t)+\mathcal{L}u(t)+c(t)u(t)=f(t),\quad \text{in }\mathcal{D}^{\prime}(0,T; H_\Omega(\R^N,K)^{\prime}),
\]
which gives in particular that $u^{\prime}\in L^{2}(0,T;H_\Omega(\R^N,K)^{\prime})$.  Choosing the function $\eta(x,t)=\varphi(x)\xi(t)$ with $\varphi \in C_{0}^{\infty}(\Omega)$ and $\xi(t)\in C_{0}^{\infty}(0,T)$ we have that $u$ satisfies relation \eqref{definition_evo}. It remains to show that $u$ verifies the initial condition. Notice that since
\[
\hat{u}_{N}^{\prime}=-\mathcal{L}u_{N}-c_{N}u_{N}+f_{N}(t)\rightharpoonup- \mathcal{L}u-cu+f=u^{\prime} \quad\text{weakly in }L^{2}(0,T;H_\Omega(\R^N,K)^{\prime}),
\]
we have that
\[
\hat{u}_{N}\rightharpoonup u \quad\text{weakly in }\mathcal{W}(0,T;H_{\Omega}(\R^N,K))
\]
and since
\[
\mathcal{W}(0,T;H_{\Omega}(\R^N,K))\hookrightarrow C([0,T];L^{2}(\Omega))
\]
with continuous injection, we have
\[
\hat{u}_{N}(0) \rightharpoonup u(0) \quad\text{weakly in }L^{2}(\Omega) 
\]
but $\hat{u}_{N}(0)=u_{0}$ hence $u(0)=u_{0}$ andd $u$ is a weak energy solution to problem \eqref{problem1_evo}.\\
\noindent In order to show that $\{v_N\}_{N\in\N}$ converges (up to a subsequence) to the weak solution $v$ to problem \eqref{problem2_evo} we use similar arguments, once we simply observe that
\[
d_{N}\rightarrow c_{\sharp}, \quad g_{N}\rightarrow f^{\sharp}\text{ strongly in }L^{2}(\Omega^{\sharp}\times (0,T)):
\]
indeed, the contractivity of the map $\eta\rightarrow \eta^{\sharp}$ from $L^{p}(\Omega)$ to $L^{p}(\Omega^{\sharp})$, $p\in [1,\infty]$ implies that
\[
\|d_{N}-c_{\sharp}\|_{L^{2}(\Omega^{\sharp}\times(0,T))}\leq \|c_{N}-c\|_{L^{2}(\Omega\times(0,T))},\quad \|g_{N}-f^{\#}\|_{L^{2}(\Omega^{\sharp}\times(0,T))}\leq \|f_{N}-f\|_{L^{2}(\Omega\times(0,T))}.
\]

\noindent{\bf Step 3.} {\it The comparison result.}\\
Let us look at the discretized elliptic problems \eqref{problem1_discr}-\eqref{problem2_discr} as $n=0,1,...,N-1$. For $n=0$,  we observe that
\[
\left(c_0(x)+\frac 1 {\Delta t}\right)_\sharp=(c_0)\sharp (x)+\frac 1 {\Delta t}=d_{0}(x)+\frac 1 {\Delta t}\quad \text{in }\Omega.\]
Moreover, we observe that, by using the Hardy inequality \eqref{HardyLit}, the following inequality holds true
\begin{equation*}
\int_{B_r}\left(f_0(x)+\frac{u_0(x)}{\Delta t}\right)^\sharp dx\le  \int_{B_r}\left(f_0^\sharp(x)+\frac{u^\sharp_0(x)}{\Delta t}\right) dx\qquad \forall r>0.
\end{equation*} 
So, from Theorem \ref{main_thm_stationary} and Remark \ref{Extensionconc}, the following inequality holds
\[
\int_{B_r}u_1^\sharp (x)dx\leq\int_{B_r}v_1(x)dx\qquad \forall r>0.
\]
Assume by induction that
\[
\int_{B_r}u_{n}^\sharp (x)dx\leq\int_{B_r}v_{n}(x)dx\qquad \forall r>0.
\]
holds for some $n\leq (N-1)$.
Therefore
\begin{equation*}
\int_{B_r}\left(f_n(x)+\frac{u_{n}(x)}{\Delta t}\right)^\sharp dx\le  \int_{B_r}\left(f_n^\sharp(x)+\frac{u^\sharp_{n}(x)}{\Delta t}\right) dx\le \int_{B_r}\left(f_n^\sharp(x)+\frac{v_{n}(x)}{\Delta t}\right) dx\qquad \forall r>0.
\end{equation*}
Thus Theorem \ref{main_thm_stationary} and Remark \ref{Extensionconc} give again
\[
\int_{B_r}u_{n+1}^\sharp (x)dx\leq\int_{B_r}v_{n+1}(x)dx\qquad \forall r>0.
\]
Taking into account the form of the piecewise constant interpolation functions $u_{N}$ and $v_{N}$, this concentration comparison can be written into the form
\begin{equation}
\int_{B_r}u_{N}^\sharp (x,t)dx\leq\int_{B_r}v_{N}(x,t)dx\qquad \text{ for all } r>0, t\in(0,T).\label{conccompapruv}
\end{equation}
Now we wish to pass to the limit as $N\rightarrow\infty$ in \eqref{conccompapruv}. By Proposition \eqref{prop_equiv_conc} we find 
\[
\int_{\Omega}u_{N}(x,t)\,\varphi(x)\,dx\leq \int_{\Omega^{\sharp}}v_{N}(x,t)\,\varphi^{\sharp}(x)\,dx,
\]
for all nonnegative $\varphi\in L^{2}(\Omega)$. Observe that $u_{N}$ converges weakly to $u$ in $L^{2}(\Omega\times (0,T))$ and $v_{N}$ converges weakly to $v$ in $L^{2}(\Omega^{\sharp}\times (0,T))$. Taking a nonnegative bounded function $\xi$ in $(0,T)$ we can write
\[
\int_{0}^{T} \int_{\Omega}u_{N}(x,t)\,\varphi(x)\xi(t)\,dx\, dt\leq \int_{0}^{T}\int_{\Omega^{\sharp}}v_{N}(x,t)\,\varphi^{\sharp}(x)\xi(t)\,dx\,dt
\]
hence we can pass to the limit as $N\rightarrow \infty$ in order to obtain
\[
\int_{0}^{T} \int_{\Omega}u(x,t)\,\varphi(x)\xi(t)\,dx\, dt\leq \int_{0}^{T}\int_{\Omega^{\sharp}}v(x,t)\,\varphi^{\sharp}(x)\xi(t)\,dx\,dt.
\]
for all $\xi\in L^{\infty}_{+}(0,T)$.
Setting
\[
\psi(t)=\int_{\Omega}u(x,t)\,\varphi(x)\,dx-\int_{\Omega^{\sharp}}v(x,t)\,\varphi^{\sharp}(x)dx
\]
the previous inequality implies
\[
\int_{0}^{T}\psi_{+}(t)\xi(t)dt=\int_{0}^{T}\psi(t)\chi_{\left\{\psi(t)\geq 0\right\}}dt\leq 0,
\]
that is $\psi_{+}(t=0$ for a.e. $t>0$, namely
\[
\int_{\Omega}u(x,t)\,\varphi(x)\,dx\leq\int_{\Omega^{\sharp}}v(x,t)\,\varphi^{\sharp}(x)dx
\]
for all $\varphi\in L^{2}_{+}(\Omega)$. Using Proposition \eqref{prop_equiv_conc} again, we find (recall that $u,\,v$ are $L^{2}$-valued maps continuous in time) 
\[
\int_{B_r}u^\sharp (x,t)dx\leq\int_{B_r}v(x,t)dx\qquad \text{ for all } r>0, t\in(0,T)
\]
and we achieve the desired  conclusion.
\quad\quad\quad\quad\quad\quad\quad\quad\quad\quad\quad\quad\quad\quad\quad\quad\quad\quad\quad\quad\qed

\section*{Acknowledgments}
B.V. wishes to thank Fernando Quir\'os and Irene Gonz\'alvez for kindly pointing the papers \cite{KasMimica, bae_kang}. 
The authors were partially supported by PRIN 2017 ``Direct and inverse problems for partial differential equations: theoretical aspects and applications'' and by Gruppo Nazionale per l'Analisi Matematica, la Probabilit\`a e le loro Applicazioni (GNAMPA) of Istituto Nazionale di Alta Matematica (INdAM).
This study was also carried out within the ``Partial differential equations and related geometric-functional inequalities'' and ``Geometric-Analytic Methods for PDEs and Applications (GAMPA)'' projects - funded by the Ministero dell'Universit\`a e della Ricerca - within the PRIN 2022 program (D.D.104 - 02/02/2022). This manuscript reflects only the authors' views and opinions and the Ministry cannot be considered responsible for them.
\nc


\bibliographystyle{siam}
\bibliography{bibliography.bib}

\

2000 \textit{Mathematics Subject Classification.}
35B45,  
35R11,   	
35J25. 


%
\textit{Keywords and phrases.} Symmetrization, nonlocal elliptic equations, nonnegative kernel.

\end{document}